\documentclass[11pt]{amsart}
\usepackage{mathtools}
\usepackage{amsmath}
\usepackage{amssymb}
\usepackage{tikz-cd}
\usepackage{yhmath}
\usepackage{graphicx}
\usepackage{ mathrsfs }
\usepackage{bbm}
\usepackage{xcolor}
\usepackage{tikz-cd}
\DeclareMathAlphabet{\mathpzc}{OT1}{pzc}{m}{it}

\newtheorem{theorem}{Theorem}[section]

\newtheorem*{claim*}{Claim}

\newtheorem{lemma}[theorem]{Lemma}
\newtheorem{lem}[theorem]{Lemma}

\newtheorem{cor}[theorem]{Corollary}
\newtheorem{proposition}[theorem]{Proposition}

\newtheorem{prop}[theorem]{Proposition}

\newtheorem{thm}[theorem]{Theorem}

\theoremstyle{definition}
\newtheorem{definition}[theorem]{Definition}\newtheorem{Def}[theorem]{Definition}
\newtheorem{example}[theorem]{Example}

\theoremstyle{remark}
\newtheorem{remark}[theorem]{Remark}
\newtheorem{Rmk}[theorem]{Remark}

\numberwithin{equation}{section}

\newcommand{\norm}[1]{\lVert#1\rVert}

\newcommand{\op}{\operatorname}

\newcommand{\bb}{\mathbb}

\newcommand{\Ga}{\Gamma}
\newcommand{\ga}{\gamma}

\newcommand{\la}{\lambda}
\newcommand{\La}{\Lambda}
\newcommand{\ba}{\backslash}

\newcommand{\cal}{\mathcal}
\newcommand{\br}{\mathbb R}

\newcommand{\G}{\Gamma}

\newcommand{\F}{\cal{F}}
\renewcommand{\frak}{\mathfrak}
\newcommand{\e}{\epsilon}
\newcommand{\mm}{\mathsf m}
\newcommand{\be}{\begin{equation}}
\newcommand{\ee}{\end{equation}}
\newcommand{\BR}{\op{BR}}
\newcommand{\inte}{\op{int}}
\renewcommand{\L}{\mathcal L}
\newcommand{\fa}{\mathfrak a}

\newcommand{\BMS}{\op{BMS}}
\newcommand{\dg}{D_\Gamma^{\star}}
\newcommand{\fg}{\mathfrak{g}}
\newcommand{\fp}{\mathfrak{p}}
\renewcommand{\i}{\op{i}}

\newcommand{\fk}{\mathfrak{k}}

\newcommand{\z}{\mathbb Z}

\renewcommand{\e}{\varepsilon}
\renewcommand{\epsilon}{\varepsilon}
\newcommand{\pc}{P^{\circ}}
\newcommand{\mc}{M^{\circ}}
\newcommand{\fc}{\mathcal F^\circ}
\newcommand{\npc}{\tilde{\nu}_\psi}

\newcommand{\E}{\mathcal E}
\newcommand{\Om}{\Omega}
\newcommand{\yg}{\mathfrak Y_\Ga}
\newcommand{\spa}{\La_{\psi}^{\spadesuit}}

\begin{document}

\title[Ergodic decompositions]{
Ergodic decompositions of geometric measures on Anosov homogeneous spaces.}

\author{Minju Lee}
\address{Mathematics department, University of Chicago, Chicago, IL  60637}
\email{minju1@uchicago.edu}
\author{Hee Oh}
\address{Mathematics department, Yale university, New Haven, CT 06520}
\email{hee.oh@yale.edu}
\begin{abstract} Let $G$ be a connected semisimple real algebraic group and
$\Gamma$  a Zariski dense Anosov subgroup of $G$ with respect to a minimal parabolic subgroup $P$.
Let $N$ be the maximal horospherical subgroup of $G$ given by the unipotent radical of $P$.
We describe the $N$-ergodic decompositions of all Burger-Roblin measures  as well as the $A$-ergodic decompositions of all Bowen-Margulis-Sullivan measures on $\Ga\ba G$.
As a consequence, we obtain the following refinement of the main result of \cite{LO}:
the space  of all {\it non-trivial} $N$-invariant ergodic and $P^\circ$-quasi-invariant
 Radon measures on $\Ga\ba G$, up to constant multiples, is homeomorphic to
 $ \br^{\text{rank}\,G-1}\times \{1, \cdots, k\}$
 where $k$ is the number of $\pc$-minimal subsets in $\Ga\ba G$.  
  \end{abstract}

\thanks{Lee and Oh respectively
supported by the NSF grants DMS-1926686 (via the
Institute for Advanced Study) DMS-1900101.
}

\maketitle

\tableofcontents

\section{Introduction}
Let $G$ be a connected  semisimple real algebraic group, i.e., the identity component of the group of real points of a semisimple algebraic group defined over $\br$. Let $\Ga<G$ be a Zariski dense Anosov subgroup of $G$ with respect to a minimal parabolic subgroup $P$.
Fix a Langlands decomposition $P=MAN$
where $N$ is the unipotent radical of $P$, $A$ is the identity component of a maximal real split torus of $G$ and $M$ is the maximal compact subgroup of $P$ commuting with $A$. The subgroup $N$ is a maximal horospherical subgroup of $G$, and in fact, any maximal horospherical subgroup of $G$ arises in this way.

In our earlier paper \cite{LO}, we showed that all $NM$-invariant Burger-Roblin measures on $\Ga\ba G$, parameterized by $ \br^{\text{rank}\,G-1}$, are $NM$-ergodic and that they describe {{precisely}} all {non-trivial} $NM$-invariant ergodic and $\pc$-quasi-invariant
 Radon (i.e., locally finite Borel) measures on $\Ga\ba G$,
 where $P^\circ$ is the identity component of $P$.  One cannot replace $NM$ by $N$ in these statements, as the Burger-Roblin measures are not $N$-ergodic in general. The main aim of this paper is to describe the $N$-ergodic decompositions of Burger-Roblin measures as well as
to classify all {non-trivial} $N$-invariant ergodic and $\pc$-quasi-invariant
 Radon measures on $\Ga\ba G$.
 When $G$ has rank one, the class of Anosov subgroups of $G$ coincides with that of convex cocompact subgroups. If $P$ is connected in addition, which is equivalent to saying $G\not\simeq \op{SL}_2(\br)$, then there exists a unique non-trivial $N$-invariant ergodic measure, as shown by Burger, Roblin and Winter (\cite{Bu}, \cite{Ro}, \cite{Win}).
 This unique measure is called the Burger-Roblin measure.
 We also mention that 
  when $\Gamma<G$ is a lattice, the classification of ergodic invariant measures for a maximal horospherical subgroup action
 was obtained by Furstenberg, Veech and Dani  (\cite{Fu},  \cite{Vee}, \cite{Da}), prior to Ratner's more general measure classification theorem for any connected unipotent subgroup action \cite{Ra}.

\medskip

We begin by recalling the definition of an Anosov subgroup. Let $\cal F:=G/P$  denote the Furstenberg boundary, and
 $\cal F^{(2)}$ the unique open $G$-orbit in $\cal F\times \cal F$. 
A Zariski dense discrete subgroup $\Ga<G$ is called an {\it Anosov subgroup} (with respect to $P$) if it is a  finitely generated word hyperbolic group which admits a $\Ga$-equivariant continuous embedding $\zeta$ of the Gromov boundary $\partial \Ga$  into $\cal F$ such that $(\zeta(x),\zeta(y))\in\cal F^{(2)}$ for all $x\ne y$ in $\partial\Ga$ (\cite{La}, \cite{GW}, \cite{KLP}, \cite{Wie}).
The class of Anosov subgroups include the Zariski dense images of representations in the Hitchin component as well as Zariski
dense Schottky subgroups.

 Denote by $\fa$ the Lie algebra of $A$ and fix a positive Weyl chamber $\fa^+\subset\fa$
so that $\log N$ is the sum of positive root subspaces. Fix a maximal compact subgroup $K$ of $G$ as in section \ref{sec.pre},
so that the Cartan decomposition
$G=KA^+ K$ holds for $A^+=\exp \fa^+$  (Def. \ref{Cartan}). 

 Let $\cal L_\Ga\subset \fa^+$ denote the limit cone of $\Gamma$ (Def. \ref{lc}), which is known to be a  convex cone with non-empty interior by Benoist \cite{Ben}.
Let $\psi_\Ga : \fa\to\bb R\cup\{-\infty\}$ be the growth indicator function  of $\Gamma$ as defined by Quint
 (Def. \ref{gi}). 
Consider the following set of linear forms on $\fa$: $$   D_\Ga^\star :=
\{\psi\in \fa^*: \psi \ge \psi_\Gamma,  \psi(v)=\psi_\Gamma(v) \text{ for some
$v\in \inte \L_\Ga$}\}.$$

For each $\psi\in \dg$, we denote by
$m^{\BR}_\psi$ and $m^{\BMS}_\psi$ respectively the Burger-Roblin measure and the Bowen-Margulis-Sullivan measure on $\Gamma\ba G$ associated to $\psi$
(see \eqref{bmss} and \eqref{brr}). The Burger-Roblin measures are all supported on
the unique $P$-minimal subset of $\G\ba G$:
$$\cal E:=\{[g]\in \G\ba G: gP\in \La\}$$
where $\La\subset \F$ denotes the limit set of $\Ga$.
In \cite{LO}, we showed that for $\Ga$ Anosov, each $m_\psi^{\BR}$ is $NM$-ergodic and the map $$\psi\mapsto m_\psi^{\BR}$$ gives a homeomorphism between $\dg$
and the space of all $NM$-invariant ergodic and $P$-quasi invariant Radon measures supported on $\cal E$, up to constant multiples. 
We also showed that all
$m^{\BMS}_\psi$, $\psi\in \dg$, are $AM$-ergodic. 

Denote by $\mathfrak Y_\Ga$ the collection of  all $\pc$-minimal subsets of $\G\ba G$. Fixing $\E_0\in \mathfrak Y_\Ga$, we set
$$P_\Ga:=\{p\in P:\E_0 p=\E_0\}.$$
By the work of Guivarc'h and Raugi \cite{GR}, the subgroup $P_\Ga$ is independent of the choice of $\E_0\in \mathfrak Y_\Ga$, and is a co-abelian subgroup of $P$ containing $\pc$.   It follows that for any $\E_0\in \mathfrak Y_\Ga$,
the map $[p]\mapsto \E_0 p$ defines a bijection between $P/P_\Ga$ and $\mathfrak Y_\Ga$. Considering the partition $\E=\bigsqcup_{\E_0\in \mathfrak Y_\Ga}\E_0$,
the following is our main theorem:
\begin{thm}\label{main} For any Anosov subgroup $\Ga<G$ and $\psi\in \dg$,
\begin{enumerate}
\item $m^{\BR}_\psi=\sum_{\cal E_0\in \mathfrak Y_\Ga} m^{\BR}_\psi|_{\E_0} $
is an $N$-ergodic decomposition;

\item $m^{\BMS}_\psi=\sum_{\cal E_0\in \mathfrak Y_\Ga} m^{\BMS}_\psi|_{\E_0 } $
is an $A$-ergodic decomposition.
\end{enumerate}
In particular,  the number of the $N$-ergodic components of $m^{\BR}_\psi$ 
as well as the $A$-ergodic components of $m^{\BMS}_\psi$
are given by $\# \mathfrak Y_\Ga=[P:P_\Ga]$, independent of $\psi$. \end{thm}
See the subsection \ref{thm1} and Theorem \ref{thm.AE} for the proofs
of (1) and (2) respectively. 

 As $\pc\subset P_\Ga$, $P_\Ga$ is of the form $M_\Ga AN$ 
where 
$$M_\Ga:=\{m\in M:\E_0 m=\E_0\}.$$

Moreover, by \cite[Prop. 4.9(a)]{BQ}, the subgroup $M_\Ga$ can be explicitly described
as follows: 
$$M_\Gamma=\text{closure of }{\{m\in M: g^{-1} h am n g \in \Gamma\text{ for some $h\in N^+, a\in A, n\in N$}}\}$$
for any $g\in G$ such that
 $g\Gamma g^{-1}\cap \inte A^+M\ne \emptyset$, where $N^+$ denotes the opposite horospherical subgroup to $N$. The subgroup $M_\Gamma$ is not equal to $M$ in general:
there exists a Zariski dense Schottky subgroup $\Gamma$ with $M_\Ga\ne M$ \cite{Be2}, and for an Anosov subgroup
$\Gamma$ which arises as the image of a Hitchin representation into $\op{PSL}_n(\br)$, it is known that $M_\Gamma=\{e\}$ \cite{La}.

Since each $\E_0\in \mathfrak Y_\Ga$ is a second countable topological space, almost all orbits are dense with respect to an ergodic measure
with full support in $\E_0$.
Hence
Theorem \ref{main} implies:
\begin{cor}\label{aeae} Let $\E_0$ be a $\pc$-minimal subset of $\Ga\ba G$. Then 
\begin{enumerate}
\item for $m_\psi^{\BR}|_{\E_0}$ almost all $x\in \E_0$,
$xN$ is dense in $\E_0$;
\item for $m_\psi^{\BMS}|_{\E_0}$ almost all $x\in \E_0$,
$xA$ is dense in $ \op{supp } m_\psi^{\BMS}\cap \cal E_0$
.
\end{enumerate} 
\end{cor}

Indeed, Corollary \ref{aeae}(2) holds for $A^+$-orbits as well (see Corollary \ref{full3}).

 In view of our earlier work \cite{LO}, Theorem \ref{main} implies:
\begin{thm} \label{mmm} The space  of all $N$-invariant ergodic and $P^\circ$-quasi-invariant
 Radon measures on $\cal E$, up to constant multiples, is given by
 $\{m_\psi^{\BR}|_{\cal E_0}: \psi\in D_\Ga^\star, \cal E_0\in \mathfrak Y_\Ga\}$ and hence homeomorphic
 to $ \br^{\text{rank}\,G-1} \times \{1, \cdots, \# M/M_\Ga\}$. 
\end{thm}

We mention a recent measure classification result \cite{LLLO} 
which is based on the above theorem.
\medskip

\subsection*{On the proofs} For each $\psi\in \dg$, there exists a unique $(\Ga, \psi)$-Patterson-Sullivan measure, say, $\nu_\psi$, on the limit set $\La\subset G/P$.
Denote  by $\npc$ the $M$-invariant lift of $\nu_\psi$ to $G/\pc$.
We first show 
that the $\Gamma$-ergodic components of $\npc$ and the $A$-ergodic components of $m_\psi^{\BMS}$ are respectively
 given by their restrictions to $\Ga$-minimal subsets of $G/\pc$ and to $\pc$-minimal subsets of $\Ga\ba G$;
 hence Theorem \ref{main}(2). 
  We define the closed subgroup, say $\mathsf E_{\nu_\psi}$
 of $AM$, consisting of all $\nu_\psi$-essential values (Definition \ref{def.ess}), and show
 that elements of the generalized length spectrum of  $\Ga$, whose $\psi$-images are sufficiently large, are contained in $\mathsf E_{\nu_\psi}$ (Proposition \ref{prop.E}).
By Proposition \ref{cor.A},
this implies that $AM^\circ$ is contained in $\mathsf E_{\nu_\psi}$, from which we deduce 
Theorem \ref{main}(1), using the $NM$-ergodicity of $m_\psi^{\BR}$.  

\bigskip

\noindent{\bf Acknowledgement} 
We would like to thank Michael Hochman for helpful conversations, especially for telling us about the reference \cite{GS}. We also thank the referee for reading the manuscript carefully and making a useful suggestion.

\section{Preliminaries}\label{sec.pre}
 Let $G$ be a connected semisimple real algebraic group and $\Gamma <G$ be a Zariski dense discrete subgroup. We fix, once and for all, a Cartan involution $\theta$ of the Lie algebra $\mathfrak{g}$ of $G$ and decompose $\fg$ as $\frak g=\frak k\oplus\mathfrak{p}$, where $\fk$ and $\fp$ are the $+ 1$ and $-1$ eigenspaces of $\theta$, respectively. 
We denote by $K$ the maximal compact subgroup of $G$ with Lie algebra $\fk$. We use the notation $o$ for the coset $[K]$ in the associated Riemannian symmetric space $G/K$. We also choose a maximal abelian subalgebra $\fa$ of $\frak p$, and set $A:=\exp \frak a$. 
Choosing a closed positive Weyl chamber $\fa^+$ of $\fa$, we also set $A^+:=\exp \frak a^+$. The centralizer of $A$ in $K$ is denoted by $M$ and we set 
$N$ to be the contracting horospherical subgroup: for  $a\in \inte A^+$,
  $N=\{g\in G: a^{-n} g a^n\to e\text{ as $n\to +\infty$}\}$.
  Note that $\log   N $ is the sum of all positive root subspaces for our choice of $A^+$.
  Similarly, we also consider the expanding horospherical subgroup $N^+$:
for  $a\in \inte A^+$,  $N^+:=\{g\in G: a^n g a^{-n}\to e\text{ as $n\to +\infty$}\}$.
We set $P=MAN$ which is a minimal parabolic subgroup of $G$.
The quotient $\F=G/P$ is known as the Furstenberg boundary of $G$ and is isomorphic to $K/M$.
We let $\La\subset \F$ denote the limit set
of $\Ga$ as defined in \cite{Ben} (see also \cite[Lem. 2.13]{LO} for an equivalent definition), which is known to be the unique $\Gamma$-minimal subset of $\F$.

 \medskip 

We fix an element $w_0$ of the normalizer of $\fa$ such that $\op{Ad}_{w_0}\fa^+= -\fa^+$. 
The opposition involution $\i:\fa \to \fa$ is defined as $\i(u)=-\op{Ad}_{w_0} u$.
\begin{Def}[Visual maps]\rm   For each $g\in G$, we define 
   $$  g^+:=gP\in G/P\quad\text{and}\quad g^-:=gw_0P\in G/P.$$
For all $g\in G$ and $m\in M$, observe that  $g^{\pm}=(gm)^{\pm}=g(e^{\pm})$.
 Let $\F^{(2)}$ denote the unique open $G$-orbit in $\F\times \F$:
$$\F^{(2)}=G(e^+, e^-)=\{(g^+, g^-)\in \cal F\times \cal F: g\in G\}.$$ 
\end{Def}

 We say that $\xi, \eta\in \F$ are in general position
 if $(\xi, \eta)\in \F^{(2)}$.

\subsection{$A$-valued cocycles}
 \begin{Def}\rm The $A$-valued Iwasawa cocycle $
 \sigma^{A}
 : G\times \F \to A$ is defined as follows: for $(g, \xi)\in G\times \F$,
let $ \sigma^A(g,\xi)\in A$ be the unique element satisfying
\be\label{exchange} gk\in K \sigma^A(g, \xi) N\ee
where $k\in K$ is such that $\xi=k^+$.\end{Def}
\begin{Def}\label{Bu0} \rm The $A$-valued Busemann function $\beta^A: \F\times G\times G\to A $ is defined as follows: for $\xi\in \F$ and $g_1, g_2\in G$, set
$$
\beta_\xi^{A} (g_1,g_2):=\sigma^{A} (g_1^{-1},\xi)\,\sigma^{A} (g_2^{-1},\xi)^{-1}.
$$
\end{Def}

\subsection{$AM$-valued cocycles}
The product map $N^+\times P\to G$ is a diffeomorphism onto its image which is Zariski open and dense in $G$. Hence
for each $\xi \in N^+e^+$, we can define $h_\xi\in N^+$ to be the unique element such that \be\label{hxi} \xi=h_\xi e^+.\ee
Similarly, the product map $K\times A\times N\to G$ is a diffeomorphism, giving the Iwasawa decomposition $G=KAN$. 
We can therefore define $k_\xi \in K$ to be  the unique element such that \be\label{kxi} h_\xi \in k_\xi AN .\ee

\begin{definition} [Bruhat cocycle and Iwasawa cocycle] \label{co} Let $g\in G$ and $\xi\in \F$ be such that
$\xi, g\xi \in N^+e^+$.
\begin{enumerate}
\item We define the Bruhat cocycle $b(g,\xi)\in AM$ to be the unique element satisfying
$$ g h_\xi \in N^+  b(g,\xi) N .$$
Note that the condition $\xi\in N^+e^+$ allows us to get $h_\xi\in N^+$ and the condition $g\xi\in N^+e^+$ implies
$gh_\xi\in N^+AMN$.
\item  We define the Iwasawa cocycle  $\sigma^{AM}(g,\xi)\in AM$ to be the unique element satisfying
$$
g k_\xi\in k_{g\xi}\sigma^{AM}(g,\xi)N.
$$
\end{enumerate}
Note that $gh_\xi\in h_{g\xi} b(g, \xi) N$.
\end{definition}

We remark that although $\log \sigma^A(g, \xi)$ was defined as the Iwasawa cocycle in \cite{LO}, we find it more convenient to use the above notation in this paper.
In order to define the $AM$-valued Iwasawa cocycle, it is necessary to choose a Borel section of the
projection $K\simeq G/AN\to K/M\simeq G/P$. In the above definition, we have used a section $\mathsf s:G/P\to G/AN$
given by $\mathsf s (hP)= hAN$ for all $h\in N^+$, so that it is continuous on $N^+e^+\subset \F$. It follows that for each fixed $g\in G$, the maps $\xi \mapsto b(g,\xi)$ and $\xi\mapsto \sigma^{AM}(g,\xi)$  are continuous
on the set $\{\xi\in N^+e^+: g\xi \in N^+e^+\}$.

\begin{definition} [$AM$-valued Busemann map] \label{def.Bu}
For $(\xi, g_1, g_2)\in \cal F \times G\times G$ such that $\xi, g_1^{-1}\xi, g_2^{-1}\xi\in N^+e^+$, we define 
$$
\beta_\xi^{AM} (g_1,g_2):=\sigma^{AM} (g_1^{-1},\xi)\sigma^{AM} (g_2^{-1},\xi)^{-1}.
$$
\end{definition}

\noindent

%

\begin{remark}\label{rmk.WD}
For fixed $g_1, g_2\in G$, the map $\xi\mapsto 
\beta_\xi^{AM}
(g_1, g_2)$ is 
continuous on the set $\{\xi\in N^+e^+:g_1^{-1}\xi, g_2^{-1}\xi\in N^+e^+\}$.
\end{remark}
We  have the following whenever both sides are defined: for any $g_1,g_2, g_3 \in G$ and $\xi\in \F$,
\begin{enumerate}
\item (cocycle identity)
$
\beta_{\xi}^{AM}
(g_1 ,g_3)=
\beta_{\xi}^{AM}
(g_1 ,g_2 )\,
\beta_{\xi}^{AM}
(g_2 ,g_3)$; 
\item (equivariance)
$
\beta_{g_3\xi}^{AM}
(g_3g_1 ,g_3g_2)=
\beta_{\xi}^{AM}
(g_1 ,g_2)$.
\end{enumerate}
We define $\beta^M$ to be the projection of $\beta^{AM}$ to $M$; we then have $\beta^{AM}_\xi(g_1, g_2)=\beta_\xi^A(g_1, g_2)\beta_\xi^M(g_1, g_2)$.
It is simple to check the following:
\begin{example} \label{ex} If $g=ham n\in N^+AMN$, then
$\beta_{g^+}^M(e,g)=m$.
\end{example}
\subsection{Jordan projection and Cartan projection} Recall that for any loxodromic element $g\in G$, there exists $\varphi\in G$ such that
$$
g=\varphi am\varphi^{-1}
$$
for some element $am\in \op{int}A^+M$.
Moreover such $\varphi$  belongs to a unique coset in $G/AM$.
We set
$$y_g:=\varphi^+\in \F$$
which is called the attracting fixed point of $g$.
The element $a\in \op{int}A^+$ is uniquely determined and  called the Jordan projection of $g$. We denote it by $\la(g)$. For a general element $g\in G$, $g$ can be written as a
commuting product $g_hg_ug_e$ where $g_h$, $g_u$ and $g_e$ are hyperbolic, unipotent and elliptic respectively. The hyperbolic element $g_h$ belongs to $AM$ up to conjugation, and the Jordan projection $\lambda(g)$ of $g$ is defined as the unique element of $\fa^+$  such that $g_h\in \varphi \exp \lambda(g) m \varphi^{-1}$
for some $\varphi\in G$ and $m\in M$.

\begin{Def} \label{lc} The limit cone $\cal L_\Ga\subset \fa^+$ is defined as  the smallest closed cone containing
all $\la(\gamma)\in \fa^+$, $\gamma\in \Ga$. \end{Def}
This is known to be a convex cone with non-empty interior \cite{Ben}.

\begin{Def} [Cartan projection]\label{Cartan} \rm
For each $g\in G$, there exists a unique element $\mu(g)\in \frak a^+$, called the Cartan projection of $g$, such that
\begin{equation*}
g\in K\exp(\mu(g))K.
\end{equation*}
\end{Def}

\section{Generalized length spectrum}

In this section, we fix a discrete Zariski dense subgroup $\Ga$ of $G$.

\subsection{$\pc$-minimal subsets of $\G\ba G$.}\label{defyy}
Since $\La$  is the unique $\Ga$-minimal subset of $\F$, it follows that the set 
\begin{equation}\label{eq.calE}
\cal E:=\{[g]\in \G\ba G: g^+\in \La\}    
\end{equation}
is the unique $P$-minimal subset of $\G\ba G$.
We refer to \cite[Thm.~2 and Thm.~1.9]{GR} for results in this subsection.  Set $\fc=G/\pc$. For any $g\in G$ with $g^+\in \La$, the closure of $\Gamma g [\pc ]$ is a $\Ga$-minimal subset of $\fc$.
Moreover the following closed subgroup of $M$ is well-defined:
\be\label{mg} M_\G:=\{m\in M: \La_0 m=\La_0\}\ee
for any $\G$-minimal subset $\La_0$ of $\fc$.
The subgroup $\mc$ is a co-abelian subgroup of $M$ and $M_\Ga/\mc$ is isomorphic to
$(\z/2\z)^{p}$ for some $0\le p\le \text{dim} A$.

For any $\Ga$-minimal subset $\La_0$ of $\F_0$,
the map $s\mapsto \La_0s$ gives a bijection between $M_\Ga\ba M$ and the collection $\cal Y_\Ga$ of all $\Ga$-minimal subsets of $\fc$.
If we set $\tilde \La:=\{g\pc\in \fc : gP\in \La\}$,
then $$\tilde \La=\bigsqcup_{\La_0\in \cal Y_\Ga}\La_0.$$

These results can be translated into statements about $\pc$-minimal subsets of $\Ga\ba G$ by duality.
Each $\La_0\in \cal Y_\Ga$ is of the form $E(\La_0)  /\pc$ for some left $\Gamma$-invariant and right 
$\pc$-invariant closed subset $E(\La_0)$ of $G$. The map $\La_0\mapsto \Ga\ba E(\La_0)$
gives a bijection between $\cal Y_\Ga$ and the collection of all $\pc$-minimal subsets of $\Ga\ba G$, say $\mathfrak Y_\Ga$.
Moreover, if we set \be\label{pg} P_\Ga:=M_\Ga AN,\ee then
$P_\Ga=\{p\in P: \cal E_0 p=\cal E_0\}$ for all $\cal E_0\in \mathfrak Y_\Ga$. We also have
$$\cal E=\bigsqcup_{\cal E_0\in \mathfrak Y_\Ga} \cal E_0.$$

 We remark that each $\pc$-minimal subset of $\Ga\ba G$ is in fact $AN$-minimal; this follows from \cite[Thm.~2]{GR}.
\subsection{Generalized length spectrum}
We  define
\begin{equation}\label{eq.GS}
\Ga^\star:=\{\ga\in \Ga : \text{there exists }\varphi\in N^+N \text{ with } \ga\in  \varphi  (\op{int}A^+M )\varphi^{-1}\}.
\end{equation}
Note that if $\ga\in\Ga$
is loxodromic 
and $y_{\ga}\in N^+e^+$, then $\ga\in\Ga^\star$.
As $\Gamma$ is Zariski dense, the set of loxodromic elements of $\Ga$ is Zariski dense in $G$ \cite{Ben}.
It follows that $\Ga^\star$ is Zariski dense in $G$ as well.

\begin{Def}\label{genJo} For $\ga\in\Ga^\star$, we define its {\it generalized Jordan projection} $\hat\la(\ga)$ to be the unique element of $\op{int}A^+M$
such that $$\ga=\varphi\hat \la(\ga)\varphi^{-1}\quad\text{for some $\varphi\in N^+N $.}$$ 
\end{Def}
\begin{Def} \label{spec} We call the following set the {\it generalized length spectrum} of $\Ga$:
$$
\hat\la(\Ga):=\{\hat\la(\ga)\in AM:\ga\in\Ga^\star\}.
$$

We denote by $$\mathsf s (\Gamma)$$ the closed subgroup of $AM$ generated by $\hat\la(\Ga)$.
\end{Def}
We refer to Remark \ref{choice} for the independence of $\mathsf s (\Gamma)$ on some choices.

\begin{lemma}\label{lem.Jd}
For all $\ga\in\Ga^\star$, we have 
$$\hat\la(\ga)=b(\ga,y_{\ga})=\beta^{AM}_{y_{\ga}}(e,\ga).$$
\end{lemma}
\begin{proof}
Since $\ga\in\Ga^\star$, we have $\ga=\varphi  \hat\la(\ga)\varphi^{-1}$ for some $\varphi=hn$, where $h\in N^+$ and $n\in N$.
Set $\xi:=y_{\ga}=\varphi^+$. In particular, $h_{\xi}=h$ and $h\in k_{\xi}AN$.
The defining relations for $b(\ga,\xi)$ and $\beta_{\xi}^{AM}(e,\ga)$ are
$$
\ga h\in h b(\ga,\xi)N \text{ and }\ga k_{\xi}\in k_{\xi}
\beta_{\xi}^{AM}
(e,\ga)N. 
$$
Now observe that
\begin{align*}
\ga h&=\varphi \hat\la (\ga)\varphi^{-1}h=hn \hat\la (\ga)n^{-1}\in h \hat\la(\ga)N \text{ and }\\
\ga k_{\xi}&=\varphi \hat\la (\ga)\varphi^{-1}k_{\xi}=k_{\xi}(k_{\xi}^{-1}h)n \hat\la (\ga)n^{-1}(h^{-1}k_{\xi})\in k_{\xi} \hat\la(\ga)N.
\end{align*}
Therefore $\hat \la(\ga)=b(\ga,\xi)=\beta^{AM}_{\xi}(e,\ga)$.
\end{proof}
For each $\xi\in\La\cap N^+e^+$, we define $b_{\xi}(\Ga)$ to be the closed subgroup of $AM$ generated by
all $b(\ga,\xi) $ where $ \ga\in\Ga $ and $\ga\xi\in N^+e^+ $.

\begin{lemma}\label{lem.xi}
 The subgroup $b_{\xi}(\Ga)<AM$ is independent of $\xi\in \La\cap N^+e^+$.
\end{lemma}
\begin{proof}
Let $\xi_1,\xi_2\in \La\cap N^+e^+$.
To show that $b_{\xi_1}(\Ga)= b_{\xi_2}(\Ga)$, it suffices to check that $b(\ga,\xi_2)\in b_{\xi_1}(\Ga)$ for 
any $\ga\in\Ga$ such that $\ga\xi_2\in N^+e^+$. Since $\La$ is $\Ga$-minimal, there exists a sequence $\ga_n\in\Ga$ such that $\lim_{n\to
\infty} \ga_n\xi_1 =\xi_2$.
Since $N^+e^+$ is open and $\xi_2, \ga \xi_2\in N^+e^+$, , we have $\ga_n\xi_1, \ga\ga_n\xi_1\in N^+e^+$ for all large $n$ and $b(\ga\ga_n,\xi_1)=b(\ga,\ga_n\xi_1)b(\ga_n,\xi_1)$.
Hence
$$
b(\ga,\xi_2)=\lim\limits_{n\to\infty}b(\ga,\ga_n\xi_1)=\lim\limits_{n\to\infty} b(\ga\ga_n,\xi_1)b(\ga_n,\xi_1)^{-1}\in b_{\xi_1}(\Ga),
$$
from which the lemma follows.
\end{proof}

By Lemma \ref{lem.xi}, we may define
$$
b(\Ga):=b_\xi (\Ga)\quad\text{for any $\xi\in \La\cap N^+e^+$}.    
$$
In the rest of this section, we assume that 
$$\Ga\cap \inte A^+ M \ne 
\emptyset
.$$

\begin{lemma}\label{lem.spec}
We have $b(\Ga)=\mathsf s (\Gamma).$
\end{lemma}
\begin{proof} 
We first claim that $b(\Ga)\subset \mathsf s (\Gamma)$.
By Lemma \ref{lem.xi}, it suffices to show that $b(\ga,e^+)\in \mathsf s (\Gamma)$ for any $\ga\in\Ga$ with $\ga e^+\in N^+e^+$.
Set $s_0:=a_0m_0\in \Ga\cap \op{int} A^+M
$.
Since $\gamma e^+$ and $e^-$ are in general position,
 for all sufficiently large $n$, $s_0^n\ga$ is a loxodromic element and $x_n:=y_{s_0^n\ga}$ converges to $e^+$ as $n\to\infty$.
Since $y_{s_0^n\ga}\in N^+e^+$, we have $s_0^n\ga\in\Ga^\star$ for all large $n$.
Now the claim follows from
\begin{align*}
b(\ga,e^+)&=\lim\limits_{n\to\infty}b(\ga,x_n)=\lim\limits_{n\to\infty} b(s_0^n,\ga x_n)^{-1}b(s_0^n\ga,x_n)\\
&=\lim\limits_{n\to\infty}\hat\la(s_0^n)^{-1}\hat\la(s_0^n\ga)\in\mathsf s (\Gamma)
\end{align*}
We next claim $\mathsf s (\Gamma)\subset b(\Ga)$.
Let $\ga\in\Ga^\star$ be arbitrary.
Note that $y_{\ga}\in N^+e^+$.
By Lemma \ref{lem.Jd}, $\hat\la(\ga)=b(\ga,y_{\ga})\in b_{y_{\ga}}(\Ga)$.
Since $b(\Ga)=b_{y_{\ga}}(\Ga)$ by Lemma \ref{lem.xi}, we have $\hat\la(\ga)\in b(\Ga)$, proving the claim.
\end{proof}

\begin{proposition}\label{prop.GR}
We have
\begin{enumerate}
\item $b(\Ga)=b(g^{-1}\Ga g)$ for all $g\in G$ with $g^\pm\in\La$;
\item $b(\Ga)$ is a co-abelian subgroup of $AM$ containing $AM^\circ$;
\item $b(\Ga)=AM_\Ga$.
\end{enumerate}
\end{proposition}
\begin{proof}
Claims (1) and (2) are proved in \cite[Thm.~1.9]{GR}. Claim (3) follows since $A\subset b(\Ga)$ by (2) and
the closure of $\{m\in M: \G\cap N^+A m N \ne \emptyset\}$ is equal to $M_\Ga$ \cite[Prop. 4.9(a)]{BQ}.
\end{proof}

Hence we deduce the following from Lemma \ref{lem.spec} and Proposition \ref{prop.GR}.
\begin{cor}\label{jordan}
We have  
$$\mathsf s (\Gamma)=AM_\Ga.$$
\end{cor}

\begin{Rmk}\label{choice} We mention that as long as $g\in G$ satisfies  $g^\pm\in\La$, we can use $\varphi\in g^{-1}N^+N^-$ and $\xi\in \La\cap g^{-1}N^+e^+$  in defining $\Ga^\star$, $\hat \la(\ga)$ and  $b_{\xi}(\Ga)$, and get the same $\mathsf s(\Ga)=AM_\Ga$ by \cite[Prop. 1.8 and Thm.~1.9]{GR}.
\end{Rmk}

\section{$A$-ergodic decompositions of BMS-measures}
As before, let $\G$ be a discrete Zariski dense subgroup of $G$.

\begin{Def}[Growth indicator function]\label{gi}\rm The growth indicator function $\psi_{\Gamma}\,:\,\fa^+ \rightarrow \br \cup\lbrace- \infty\rbrace$  is defined as follows:
  for any vector $u\in \fa^+$,
\begin{equation*}
\psi_{\Gamma}(u):=\|u\| \cdot \inf_{\underset{u \in\cal C}{\mathrm{open\;cones\;}\cal C\subset \fa^+}}\tau_{\cal C}
\end{equation*}
where $\tau_{\cal C}$ is the abscissa of convergence of the series $\sum_{\ga\in\Ga, \mu(\ga)\in\cal C}e^{-t\norm{\mu(\ga)}}$.
\end{Def}
We consider $\psi_\Ga$ as a function on $\fa$ by setting $\psi_\Ga=-\infty$ outside of $\fa^+$.

For a linear form $\psi\in \fa^*$, a Borel probability measure $\nu$ on $\La$ is called a $(\Ga,\psi)$-Patterson-Sullivan measure if 
 for all $\ga\in \Ga$ and $\xi\in \cal F$,
 \be\label{gc000} \frac{d\ga_* \nu}{d\nu}(\xi) =e^{\psi (\log \beta_\xi^A (e, \gamma ))} .\ee

Set $$   D_\Ga^\star :=
\{\psi\in \fa^*: \psi \ge \psi_\Gamma,  \psi(u)=\psi_\Gamma(u) \text{ for some
$u\in \inte \L_\Ga$}\}.$$
For each linear form $\psi\in \dg$, Quint constructed a $(\Gamma, \psi)$-Patterson-Sullivan measure,  say, $\nu_\psi$  \cite[Thm.~4.10]{Quint2}.
For an Anosov group $\Ga$, it was shown in \cite[Thm.~1.3]{LO} that the map $\psi \mapsto \nu_\psi$ is a homeomorphism between
$\dg$ and the space of all $\Gamma$ Patterson-Sullivan measures.

\subsection{Antipodality of $\Gamma$}
When  $\Gamma$ is Anosov, we have the following so-called antipodal property from its definition: 
\be\label{ant2} \{(\xi, \eta)\in \La\times \La: \xi\ne \eta\}\subset\F^{(2)}.\ee 
\begin{lem} \label{ant} Let $\Ga$ be Anosov. 
If $g\in G$ satisfies $g^-\in \La$, then $g^{-1}\La\subset N^+e^+\cup \{e^-\}$.
\end{lem}
\begin{proof} 
Suppose that  $\xi\in \La$ and $g^{-1}\xi\ne e^-$. Then $\xi\ne g^-$ in $\La$. Hence by \eqref{ant2}, $(\xi, g^-)\in \cal F^{(2)}$, or equivalently,
$(g^{-1}\xi, e^-)\in \cal F^{(2)}$. Since $\{\eta\in \F: (\eta, e^-)\in \F^{(2)}\}=
N^+e^+$,
$g^{-1}\xi\in N^+e^+$, proving the claim.
\end{proof}

\begin{cor}\label{full1} 
Let  $\psi\in \dg$. For any  $g\in G$ with $g^{\pm}\in \La$,
$$\nu_\psi(\La\cap gN^+e^+)=1.$$
\end{cor}

\begin{proof}
By Lemma \ref{ant}, $\La-\{g^-\}= \La\cap gN^+e^+$.
Hence the claim follows from the fact that $\nu_\psi$ is atom-free \cite[Lem. 7.8]{LO}. \end{proof}


In the rest of this section, we assume that $\G<G$ is an Anosov subgroup. We will assume that
$$\Ga\cap 
\op{int} A^+ M\ne \emptyset
;$$
this can be achieved by replacing $\Gamma$ by one of its conjugates, and hence we do not lose any generality of our discussion by making such an assumption.

By Corollary \ref{full1}, this assumption implies that  $$\nu_\psi(\La\cap N^+e^+)=1\quad\text{ for any $\psi\in \dg$.}$$

\subsection{Hopf parametrization of $G$}
The map $\op{i}(gM)= (g^+, g^-, \beta_{g^+}^{A} (e, g))$ gives a 
$G$-equivariant homeomorphism between $G/M$ and $\F^{(2)}\times A$, where
the $G$-action on the latter is given by
$$g. (\xi, \eta, a)= (g\xi, g\eta, \beta_{g\xi}^{A} (e, g) a)\quad\text{for $g\in G$ and }((\xi, \eta), a)\in \F^{(2)}\times A.$$

For the principal $M$-bundle $G\to G/M$, we fix a Borel section $\mathsf s:G/M\to G$ so that $\mathsf s(han M)=han$ for all $han\in N^+AN$.
Now for any $g\in G$, there exists a unique $m_g\in M$ such that  $g=\mathsf s(gM) m_g $.
Then the map $\op{j} (g)= (\op{i} (gM), m_g)$ gives a $G$-equivariant
Borel isomorphism of $G$ with $\F^{(2)}\times AM$
where the $G$ action on the latter is given by
\begin{equation}\label{eq.BC}
g. (\xi, \eta, am)= (g\xi, g\eta, \beta_{g\xi}^{AM} (e, g) am)
\end{equation}
whenever $\xi, g\xi\in N^+e^+$.
We call this map the Hopf parametrization of $G$ (relative to the choice of $\mathsf s$).
We mention that this map was also considered in \cite{Dang}.

The restriction of $\op{j}$ to $N^+P$ is given by
\begin{equation}\label{eq.j0}
    \op{j}(g)= (g^+, g^-, \beta_{g^+}^{AM} (e, g))\quad\text{for $g\in N^+P$}
\end{equation}
which gives
a homeomorphism 
$$N^+P\simeq 
\{(\xi,\eta,am)\in\cal F^{(2)}\times AM:\xi\in N^+e^+\}.$$

Fix $\psi\in\dg$ in the rest of this section. 
For $(\xi_1,\xi_2)\in\cal \F^{(2)}$, define the $\psi$-Gromov product:
\be \label{gro}
[\xi_1,\xi_2]_{\psi} :=\psi(\log \beta^A_{g^+}(e,g)+\op i\log \beta^A_{g^-}(e,g) )
\ee
where $g\in G$ is such that $g^+=\xi_1$ and $g^-=\xi_2$.

In terms of the Hopf parametrization of $G$, the following  defines a left $\Gamma$-invariant and right $AM$-invariant measure on $G$:
\begin{align} \label{bmss}
d\tilde m^{\BMS}_{\psi}(g)&=e^{\psi (\log\beta^A_{g^+}(e, g) +\op{i} \log\beta^A_{g^-} (e, g )) } \;  d\nu_{\psi} (g^+) d\nu_{\psi\circ \i}(g^-)\,da\, dm\\ &=
e^{[\xi_1,\xi_2]_{\psi} } \;  d\nu_{\psi} (g^+) d\nu_{\psi\circ \i}(g^-)\,da\, dm .\notag
\end{align}
We denote by  $m^{\BMS}_{\psi}$  the measure on $\Ga\ba G$ induced by  $\tilde m^{\BMS}_{\psi}$ and call it
the Bowen-Margulis-Sullivan measure (associated to $\psi$). Note that its support is equal to  
\be\label{ome} \Omega:=\{x\in \Gamma\ba G: x^{\pm}\in \La\}.\ee
In (\cite{Sambarino}, \cite{LO}), it was noted that $m_\psi^{\BMS}$ is an $AM$-ergodic  measure and that it is infinite whenever $\op{rank}G\ge 2$. 

Similarly, the Burger-Roblin measure $m_\psi^{\BR}$ on $\Ga\ba G$ is induced from the following
  left $\Gamma$-invariant and right $NM$-invariant measure on $G$: \be\label{brr}
d\tilde m^{\BR}_{\psi}(g)=e^{\psi (\log\beta^A_{g^+}(e, g)) +2\rho{( \log\beta^A_{g^-} (e, g )) }} \;  d\nu_{\psi} (g^+) dm_o(g^-)\,da\, dm,
\ee
where $\rho$ denotes the half sum of all positive roots with respect to $\fa^+$ and $m_o$ denotes the $K$-invariant probability measure on $G/P$.
Note that the support $m_\psi^{\BR}$ is equal to  
$\cal E$, which was defined in \eqref{eq.calE}.  This was first defined in \cite{ELO}.

By Corollary \ref{full1}, $$\tilde m_\psi^{\BMS}(G- N^+P)=0=\tilde m_\psi^{\BR}(G- N^+P).$$

\subsection{Ergodic decomposition of $m_\psi^{\BMS}$.} Recall  from subsection \ref{defyy}:
 $$\tilde \La=\bigsqcup_{\La_0\in \cal Y_\Ga} \La_0
 \quad \text{and}\quad  \cal E=\bigsqcup_{\cal E_0\in \mathfrak Y_\Ga} \cal E_0.$$
We denote by $\tilde \nu_\psi$ the $M/\mc$-invariant lift of $\nu_\psi$ to $\tilde \Lambda\subset \fc$, i.e., for $f\in C(\fc)$,
$$\npc(f):=\nu_\psi (\sum_{m\in M/\mc} m.f)=\nu_\psi (\int_{m\in M} m.f \,dm)$$ where $m.f (x)=f(xm)$. 

\begin{thm}\label{thm.AE}  Let $\G<G$ be an Anosov subgroup.  
\begin{enumerate}
\item The restriction $\npc$ to each $\Gamma$-minimal subset of $\fc$ is $\G$-ergodic. 
In particular, $\npc=\sum_{\La_0\in \cal Y_\Ga} \npc|_{{\La_0}}$
is a $\Ga$-ergodic decomposition.
\item The restriction of $m_\psi^{\BMS}$ to each $\pc$-minimal subset of $\Ga\ba G$ is $A$-ergodic.
\end{enumerate}
In particular, $$m_{\psi}^{\BMS}=\sum_{\cal E_0\in \mathfrak Y_\Ga} m_\psi^{\BMS}|_{{\cal E_0}}$$
is an $A$-ergodic decomposition.
\end{thm}

The rest of this section is devoted to the proof of this theorem. 
Set $$\tilde\Omega:=\{g\in G : \Ga g\in\Omega\}=\{g\in G: g^{\pm}\in \La\}.$$
Let $\cal B$ denote the Borel $\sigma$-algebra on $G$.
 We set
$$
\Sigma_\pm:=\{B\cap\tilde \Omega : B\in\cal B\text{ with  }B=\Ga B A N^\pm\}.$$
We also define $\Sigma$ to be the collection of all $B\in \cal B$ such that
 $m_{\psi}^{\BMS}(B\bigtriangleup B_+)=
m_{\psi}^{\BMS}(B\bigtriangleup B_-)=0$
 for some $B_\pm\in \Sigma_\pm$.
Recall the subgroup $M_\Ga<M$ given in \eqref{mg}, and define
$$
\Sigma_0:=\{B\cap\tilde\Omega : B\in\cal B\text{ with }B=\Ga BA M_\Ga\}.
$$
The following is a main technical ingredient of the proof of Theorem \ref{thm.AE}:
\begin{lemma}\label{lem.sub}
We have $\Sigma\subset\Sigma_0$ mod $m_{\psi}^{\BMS}$; that is, for all $B\in\Sigma$, there exists $B_0\in\Sigma_0$ such that $m_{\psi}^{\BMS}(B \bigtriangleup B_0)=0$.
\end{lemma}
This lemma follows if we show that any  bounded $\Sigma$-measurable function on $\tilde\Omega$ is $\Sigma_0$-measurable
modulo $m_\psi^{\BMS}$.

Let $f$ be any bounded $\Sigma$-measurable function on $\tilde\Omega$.
We may assume without loss of generality that $f$ is strictly left $\Ga$-invariant and right $A$-invariant \cite[Prop. B.5]{Zi}.
There exist bounded $\Sigma^\pm$-measurable functions $f_\pm$ such that $f=f_\pm$ for $m_{\psi}^{\BMS}$-a.e.
Note
that $f_\pm$ satisfy $f_\pm(gn)=f_\pm(g)$ whenever $g, gn\in\tilde\Omega$ with $n\in N^\pm$.
Set
$$
E:=\left\{ gAM :
\begin{array}{c}
f|_{gAM}\text{ is measurable and}\\
f(gm)=f_+(gm)=f_-(gm)\\
\text{for Haar a.e. }m\in M
\end{array}
\right\}\subset\tilde\Omega/AM.
$$
By Fubini's theorem, $E$ has a full measure on $\tilde\Omega/AM\simeq\La^{(2)}$ with respect to the measure $d\nu_\psi\,d\nu_{\psi\circ\op{i}}$. For all small $\e>0$, define functions $f^\e, f_\pm^\e : \tilde\Omega\to\bb R$ by
$$
f^\e(g):=\tfrac{1}{\op{Vol}(M_\e)}\int_{M_\e}f(gm)\,dm\text{ and }f_\pm^\e(g):=\tfrac{1}{\op{Vol}(M_\e)} \int_{M_\e}f_\pm(gm)\,dm
$$where $M_\e$ denotes the $\e$-ball around $e$ in $M$.
Note that if $gAM\in E$, then $f^\e$ and $ f_\pm^\e$ are continuous and identical on $gAM$.
Moreover, as $M$ normalizes subgroups $A$ and $N^{\pm}$,
 $f^\e$ is strictly left $\Ga$-invariant, right $A$-invariant and $f_\pm^\e(gn)=f_\pm^\e(g)$ whenever $g, gn\in\tilde\Omega$ with $n\in N^\pm$.
Using the isomorphism between $\tilde\Omega/AM$ and $\La^{(2)}$ given by $gAM\mapsto (g^+, g^-)$,
we may consider $E$ as a subset of  $\La^{(2)}$.
We then define
\begin{align*}
E^+:&=\{\xi\in\La : (\xi,\eta') \in E\;\;\text{ for $\nu_{\psi\circ\op{i}}$-a.e. }\eta'\in\La\};\\
E^-:&=\{\eta\in\La : (\xi',\eta) \in  E\;\; \text{ for $\nu_{\psi}$-a.e. }\xi'\in\La\}.
\end{align*}
 
Then $E^-$ is $\nu_{\psi\circ\i}$
-conull and $E^+$ is 
$\nu_\psi$-conull by Fubini's theorem.
Set
$$
E_\eta^+:=\{\xi\in\La : (\xi,\eta) \subset E\}\quad \text{ and } \quad E_\xi^-:=\{\eta\in\La : (\xi,\eta) \subset E\}.
$$

Note that $E_\xi^-$
is $\nu_{\psi\circ\i}$-conull for all 
$\xi\in E^+$ and  that $E_\eta^+$ is $\nu_\psi$-conull for all $\eta\in E^-$.
\begin{lemma}\label{lem.m0}\label{lem.m1}
Let $g\in\tilde\Omega$ be such that $gAM\in E$ and $g^\pm\in E^\pm$.
Then for any $\e>0$, $f^\e(gm_0)=f^\e(g)$ for all $m_0\in M_\Ga$.
\end{lemma}
\begin{proof} 
We will use the following observation in the proof.
For $am\in AM$, suppose that there exist $\ga\in \Ga$,
and a sequence $h_1, \cdots, h_k \in N\cup N^+ $ such that 
$\ga g am =g h_1\cdots h_k$ and
$
gh_1\cdots h_i\in E \,\,\,\text{for all}\,\,\, 1\leq i\leq k.
$
Then
$$f^\e(gam)=f^\e(\ga gam)=f^\e(gh_1\cdots h_r)=f^\e(gh_1\cdots h_{r-1})=\cdots =f^\e(g), $$ by  the $N^{\pm}$-invariance of $f_{\pm}^\e$, the invariance of $f$ by $\Gamma$ and $A$ and the fact that all three agree on $E$. 

By 
Proposition \ref{prop.GR},
it suffices to prove 
that $$f^\e(g b(g^{-1}\ga g,\xi))=f^\e(g)
$$ for any $\ga \in \Ga$ and $\xi\in g^{-1}\La\cap N^+e^+$.
Setting $b(g^{-1}\ga g,\xi)=(am)^{-1}$, we may write $\ga g am =gh_1n_1h_2$ where  $h_1,h_2\in N^+$ and $n_1\in N$.
Note that $E^\pm$ are $\Ga$-invariant, as the measures $\nu_\psi$ and $\nu_{\psi\circ\i}$ are $\Ga$-quasi-invariant. 
Since $g^\pm\in E^\pm$, we get $\ga g^\pm\in E^\pm$.
Set 
\begin{align*}
\xi_0&=g^+,&\eta_0&=g^-,\\
\xi_1&=gh_1^+,&\eta_1&=gh_1n_1^-(=\ga g^-),\\
\xi_2&=gh_1n_1h_2^+(=\ga g^+).& &
\end{align*}

Choose a sequence $\xi_{1,\ell}\in E^+\cap E_{\eta_0}^+\cap E_{\eta_1}^+$ which converges to 
$\xi_1$ as $\ell\to\infty$.
This is possible because $E^+\cap E_{\eta_0}^+\cap E_{\eta_1}^+$ is dense in $\La$, as it is $\nu_\psi$-conull from the hypothesis that $\xi_0=g^-\in E^-$ and $\xi_1=\ga g^-\in E^-$.
Let $h_{1,\ell}\in N^+$ be the unique element such  that $(gh_{1,\ell})^+=\xi_{1,\ell}
$, $n_{1,\ell}\in N$ the unique element such that $(gh_{1,\ell}\, n_{1,\ell})^-=\ga g^-$, and finally $h_{2,\ell}\in N^+$ the unique element such that $(gh_{1,\ell}\, n_{1,\ell}\, h_{2,\ell})^+=\ga g^+$.
Since $(gh_{1,\ell}\, n_{1,\ell}\, h_{2,\ell})^\pm=\ga g^\pm$, we have $gh_{1,\ell}\, n_{1,\ell}\, h_{2,\ell}=\ga g a_\ell m_\ell$ for some $a_\ell\in A$ and $m_\ell\in M$.
Note that $a_\ell m_\ell\to am$ as $\ell\to\infty$ and that  $a_\ell m_\ell\in  b(g^{-1}\Ga g)$. The  sequences $h_{1,\ell},n_{1,\ell},h_{2,\ell}\in N\cup N^+$  satisfy
\begin{itemize}
\item
$gh_{1,\ell} AM\in E$, as $(gh_{1,\ell})^-=
\eta_0
$ and $(gh_{1,\ell})^+=
\xi_{1,\ell}
\in E_{\eta_0}^+
$;
\item
$gh_{1,\ell} \, n_{1,\ell} AM\in E$, as $(gh_{1,\ell}\, n_{1,\ell})^-=\eta_1
$ and $(gh_{1,\ell}\, n_{1,\ell})^+=\xi_{1,\ell}\in E_{\eta_1}^+
$;
\item
$gh_{1,\ell}\, n_{1,\ell} \, h_{2,\ell} AM=\ga gAM \in E$, as $gAM\in E$ and $E$ is $\Ga$-invariant. 
\end{itemize}
Therefore, 
$f^\e(g a_\ell m_\ell)=f^\e(g)$ by the observation
made in the beginning of the proof.
Since $gAM\in E$, $f^\e$ is continuous on $gAM$ and hence
$$
f^\e(gam)=\lim\limits_{\ell\to\infty}f^\e(ga_\ell m_\ell)=f^\e(g).
$$
This finishes the proof.
\end{proof}

\noindent
\textbf{Proof of Lemma \ref{lem.sub}:}
Let $f$ be any bounded $\Sigma$-measurable function on $\tilde\Omega$.
For any $\e>0$, by Lemma \ref{lem.m1}, $f^\e$ coincides with a $\Sigma_0$-measurable function $m_{\psi}^{\BMS}$-a.e.
Since $\lim_{\e\to 0} f^\e= f$ $m_{\psi}^{\BMS}$-a.e., $f$ is a $\Sigma_0$-measurable function $m_{\psi}^{\BMS}$-a.e. as well.
This proves the lemma.$\qed$

\medskip

\begin{cor}\label{cor.dec}
There exists $B\in\Sigma$ such that any two distinct subsets in $\{B.s: s\in M_\Ga \ba M\}$ are measurably disjoint and
$\Sigma$ is 
the
finite $\sigma$-algebra generated by $\{B.s: s\in M_\Ga \ba M\}$  mod $m_{\psi}^{\BMS}$. 
\end{cor}
\begin{proof}
First, note that the $AM$-ergodicity of $m_{\psi}^{\BMS}$  implies that
the $\sigma$-algebra
$$
\Sigma_1:=\{B\cap\tilde\Omega : B\in\cal B\text{ such that }B=\Ga BA M\}
$$
is trivial mod $m_{\psi}^{\BMS}$. It follows that for any $B\in\Sigma_0$, and hence for any $B\in\Sigma$ by Lemma \ref{lem.sub}, with $m_{\psi}^{\BMS}(B)>0$, the union $\cup_{s\in M_\Ga\ba M}B.s$ is $m_{\psi}^{\BMS}$-conull.

Let $\cal P=\{A_1,\cdots,A_k\}$ be a partition of $\tilde \Omega$ with maximal $k$, among all partitions of $\Omega$ satisfying
\begin{enumerate}
\item
$A_i\in\Sigma$ and $m_{\psi}^{\BMS}(A_i)>0$,
\item
$\tilde\Omega=A_1\cup\cdots\cup A_k$ mod $m_{\psi}^{\BMS}$ and
\item
for any $s\in M_\Ga\ba M$, we have $A_i.s\in\{A_1,\cdots,A_k\}$ mod $m_{\psi}^{\BMS}$.
\end{enumerate}
It remains to set $B=A_1$ to prove the claim.
\end{proof}

\subsection{$\br$-ergodic decomposition of $\hat m_\psi$ on $\La^{(2)}\times\bb R\times M$.}\label{subsec.m}
Set $\La^{(2)}=(\La\times \La) \cap \F^{(2)}$. The action of $\Ga$ on $\La^{(2)}\times\bb R$ defined by
$$
\ga.(\xi,\eta,t)=(\ga\xi,\ga\eta,t+\psi(\log \beta^A_{\ga\xi}(e,\ga)))
$$
 is proper and cocompact, and  the measure 
$
d\tilde m_\psi:=e^{[\cdot,\cdot] _{\psi}}d\nu_\psi\,d\nu_{\psi\circ\op i} \,dt
$
on $\La^{(2)}\times\bb R$  descends to a finite $\br$-ergodic measure $m_\psi$ on $\Ga\ba\La^{(2)}\times\bb R$ 
(\cite[Thm.~3.2]{Samb2}, \cite[Thm.~A.2]{Car}). 
We denote by $d\hat m_\psi$ the finite measure on 
$$Z:=\Ga\ba\La^{(2)}\times\bb R\times M$$ induced by the $\Ga$-invariant
product measure $d\tilde m_\psi\,dm$ on $\La^{(2)}\times\bb R\times M$; here
 $\Ga$ acts on $\La^{(2)}\times\bb R\times M$ by
$$
\ga.(\xi,\eta,t,m)=(\ga\xi,\ga\eta,t+\psi(\log \beta_{\ga \xi}^A(e,\ga)),\beta^M_{\ga \xi}(e, \ga)m)
$$
where $(\xi, \eta)\in \La^{(2)}$, $t\in \br$ and $m\in M$.

Define the Borel map $\Psi : \tilde\Omega\to \La^{(2)}\times \bb R\times M$ by 
$$
\Psi(g)=(g^+, g^-,\psi(\beta_{g^+}^A(e,g)),\beta^M_{g^+}(e,g)).
$$
 Note that  for all $\ga\in\Ga$, $a\in A$ and $m\in M$,  $\Psi(\ga g am)=\ga\Psi(g)\tau_{\psi(\log a)}\tau_m$ 
 for $\tilde m_\psi^{\BMS}$-almost all $g\in \tilde
 \Omega$, where $\tau$ stands for the
 right translation action by elements of $\bb R\times M$.
By abuse of notation, let $\Psi : \Omega \to Z$ denote the  map induced by $\Psi$ and $\tau$ denote the action of $\bb R\times M$ on $Z$ induced by $\tau$.

Recalling that $\Omega=\bigsqcup_{\mathcal E_0\in \mathfrak Y_\Ga} (\Om\cap \E_0)$,
we set 
$$Z_{\E_0}:=\Psi (\Omega\cap \E_0)\quad\text{ for each $\E_0\in \mathfrak Y_{\Ga_0}$.}$$
Hence the collection $\{Z_{\E_0}:\E_0\in \yg\}$ gives a measurable partition for $(Z, \hat m_\psi)$.

\begin{prop} \label{lem.AI}
For each $\E_0\in \yg$,
the restriction $\hat m_\psi |_{Z_{\E_0}}$ is $\br$-ergodic, and
$\hat m_\psi=\sum_{\E_0\in \yg} \hat m_\psi |_{Z_{\E_0}}$ is an $\br$-ergodic decomposition.
In particular,
 $\npc|_{\La_0}$ is $\Ga$-ergodic and $\npc=\sum_{\La_0\in \cal Y_\Ga} \npc |_{{\La_0}}$ is a $\Ga$-ergodic decomposition.
 \end{prop}
\begin{proof} By Corollary \ref{cor.dec}, $\Sigma$ is generated by $\{B.s : s\in M_\Ga\ba M\}$ mod $m_{\psi}^{\BMS}$ for some $B\in\Sigma$. We first claim that $\hat m_\psi|_{\Psi (B.s)}$ is $\br$-ergodic for each $s\in M_\Ga\ba M$.

Let $f\in C(Z)$ be arbitrary.
The Birkhoff average $f_\sharp : Z\to\bb R$ is defined $\hat m_\psi$-a.e. by
$$
f_\sharp(y):=\lim_{T\to\infty}\frac{1}{T}\int_0^T f(y\tau_t)\,dt=\lim_{T\to\infty}\frac{1}{T}\int_{0}^T f(y\tau_{-t})\,dt.
$$
Note that $f_\sharp$ is well defined by the Birkhoff ergodic theorem and is ${\bb R}$-invariant.
Hence, $f_\sharp\circ\Psi$ is defined $m_{\psi}^{\BMS}$-a.e. The desired ergodicity follows from the Birkhoff ergodic theorem
if we show  that $f_\sharp\circ \Psi$ is constant $m_{\psi}^{\BMS}$-a.e. on each $B.s$.
Let $u\in \inte \L_\Ga$ be the unique vector such that $\psi(u)=\psi_\Ga(u)=1$ and let $a_t=\exp tu$. Observing that $f\circ \Psi$ is uniformly continuous on each $xAN\cap \Omega$ whenever $\Psi$ is continuous at $x$
and that $f(\Psi(x)\tau_t)=f(\Psi(xa_t))$ for all $t\in \br$, it is a standard Hopf argument to show that $f_\sharp\circ\Psi$ coincides with $N^\pm$-invariant functions 
$m_{\psi}^{\BMS}$-a.e.  Hence $f_\sharp\circ\Psi$ is $\Sigma$-measurable, implying that
$f_\sharp\circ \Psi$ is constant $m_{\psi}^{\BMS}$-a.e. on each $B.s$.
Therefore this proves the claim.

For each $\E_0\in \yg$, $ \hat m_\psi (\Psi(B.s)\cap Z_{\E_0})>0$  for some $s\in M_\Ga\ba M$.
It follows from the $\br$-ergodicity of $\hat m_\psi |_{\Psi (B.s)} $ that $\hat m_\psi |_{\Psi (B.s)}=\hat m_\psi|_{Z_{\E_0}}$.
Therefore the proposition is proved.
\end{proof}

The measure $m_\psi^{\BMS}$ disintegrates over $\hat m_\psi$ via the projection $\Ga\ba \La^{(2)}\times A\times M\to \Ga\ba \La^{(2)}\times \bb R\times M$, where each conditional measure is the Lebesque measure on $\exp (\op{ker}\psi)$.

\noindent
\textbf{Proof of Theorem \ref{thm.AE}.}
Since $d m_\psi^{\BMS} |_{\E_0}=d\hat m_\psi|_{Z_{\E_0}} \,d\op{Leb}_{\op{ker}\psi}$, the 
$\br$-ergodicity of  $\hat m_\psi|_{Z_{\E_0}}$ proved in Proposition \ref{lem.AI} implies the $A$-ergodicity of
 $m_\psi^{\BMS} |_{\E_0} $. $\qed$

\subsection{The set of strong Myrberg limit points}
In \cite{LO}, we defined Myrberg limit points of $\Ga$.

\begin{Def} We now define the set of {\it strong} Myrberg limit points as follows:
 \begin{multline}\label{My} \spa=\{\xi\in \La\cap N^+e^+: \text{for each $\E_0\in \yg$, there exist} \\ \text{
 $\eta\in \La$ and  $m\in M$ such that $Z_{\E_0}=\overline{\Gamma (\xi, \eta, 0, m)\br_{+}}$}  \}. \end{multline} 
\end{Def}

Since $\hat m_\psi|_{Z_{\E_0}}$ is $\bb R$-ergodic and finite for each $\E_0\in \yg$, the Birkhoff ergodic theorem for the $\br$-action implies:
\begin{cor} \label{full} We have $\nu_\psi(\spa)=1$.
\end{cor}

The same proof as the proof of \cite[Prop. 8.2]{LO}
shows that if $g\in \E_0$ and $g^+\in \spa$,
$$\limsup \Gamma\ba \Gamma gA^+=\Omega\cap \E_0.$$
Hence Corollary \ref{full} implies  (cf. \cite[Coro 8.12]{LO}):
\begin{cor}\label{full3}
For $m_\psi^{\BMS}|_{\E_0}$-almost all $x\in \E_0\cap \Omega$,
each $xA^+$ and $xw_0 A^+$ is dense in  $\E_0\cap \Omega$.
\end{cor}

Let $\Pi$ denote the set of all simple roots of $\frak g$ with respect to $\fa^+$.
\begin{Def} \rm
 For a sequence $a_n\in A^+$, we write
 $a_n\to \infty$ regularly in $ A^+$ or $\log a_n\to \infty$ regularly in $\fa^+$, if  $\alpha(\log a_n)\to\infty$ as $n\to\infty$ for all $\alpha\in\Pi$.  \end{Def}

The following is an important property of Anosov groups:
\begin{lem}\label{Anre}
Let $\Ga$ be Anosov. For any $g, h\in G$ and a sequence $\ga_n \to \infty$ in $\Ga$, $\mu(g\ga_n h)\to \infty$ regularly in $A^+$.
\end{lem}
This lemma is a consequence of the fact that the limit cone of $\Gamma$ is contained in $\inte \fa^+ \cup\{0\}$ (cf.
\cite[Thm.~4.3]{LO} for references).

 In the Cartan decomposition $g=k_1 (\exp \mu(g)) k_2\in KA^+K$,
 if $\mu(g)\in \inte \fa^+$, then $k_1, k_2\in K$ are determined uniquely up to mod $M$, more precisely,
if $g=k_1'(\exp \mu(g)) k_2'$, then there exists $m\in M$ such that $k_1=k_1'm$ and $k_2=m^{-1}k_2'$. We write 
$$\kappa_1(g):=[k_1]\in K/M\quad \text{ and } \quad \kappa_2(g):=[k_2]\in M\ba K.$$

\begin{Def}\rm
Let $o=[K]\in G/K$ and let $g_n\in G$ be a sequence. A sequence $g_n(o) \in G/K$ is said to converge to $\xi\in \cal F$ 
if $\mu(g_n) \to\infty$ regularly in $\fa^+$ and $\lim\limits_{n\to\infty}\kappa_1(g_n)=\xi$; we write
$\lim_{n\to \infty} g_n(o)=\xi$.
\end{Def}

Recall the map $\op{j}$ from \eqref{eq.j0}:
\begin{lem}\label{lem.flat} Let $\cal E_0\in\frak Y_\Ga$ and
$\tilde{\cal E}_0\subset G$ be its $\Gamma$-invariant lift.
There exists $s_0\in M/M_\Ga$ such that
$$\op{j}(\tilde \Omega\cap \tilde {\cal E}_0\cap N^+P)=
\{(\xi,\eta,ams_0)\in\La^{(2)}\times AM  : \xi\in N^+e^+,am\in AM_\Ga \}.$$
\end{lem}
\begin{proof}
Recall that $\Ga\cap\op{int}A^+M\neq\emptyset$ and hence $e^\pm\in\La$.
In particular, $\op{j}(\tilde\Om\cap\tilde{\cal E}_0\cap N^+P)$ contains an element of the form $(e^+,e^-,s_0)\in \La^{(2)}\times AM$ for some $s_0\in M$. 
Note that for all $\ga\in\Ga\cap  N^+P$, we have
$$
\ga.(e^+,e^-,s_0)=(\ga^+,\ga^-,\beta_{e^+}^{AM}(\ga^{-1},e)s_0).
$$

Since $\Ga\cap \inte A^+M\ne\emptyset$,
 $M_\Ga$ is equal to the closure of $\{m\in M: \G\cap N^+ m AN \ne \emptyset\}$ by \cite[Prop. 4.9(a)]{BQ}. Recall also that
for $\ga\in \Ga\cap N^+ mAN$,
$\beta_{e^+}^M(\ga^{-1},e)=m$.
Therefore, using the fact that $\tilde{\cal E}_0$ is right $M_\Ga AN$-invariant, we deduce that the set $\op{j}(\tilde \Omega\cap \tilde {\cal E}_0\cap N^+P)$ contains
$$\{(\ga^+,\eta, am s_0)\in \La^{(2)}\times AM:\ga\in\Ga\cap N^+P, am\in AM_\Ga \}.
$$
This proves the claim, since
$\{\ga^+\in \F:\ga\in \G\cap N^+P\}$ is dense in $\La$.
\end{proof}
\begin{lemma}\label{lem.C2} Let $p\in G/K$ and $\eta\ne \xi_0\in \La$.
For any $\xi\in \spa -\{\eta\}$, there exists an infinite sequence $\ga_i\in\Ga$ such that
\be\label{pp} \lim_{i\to \infty} \ga_i^{-1}p= \eta, \;\; \lim_{i\to \infty} \ga_i^{-1}\xi= \xi_0, \;\; \text{ and } \;\; \lim_{i\to \infty}
\beta_{
\xi
}^M(\ga_i,e)= e .\ee
Moreover, there exists a neighborhood $U$ of $\xi_0$ such that, as $i\to\infty$, the sequence $\ga_i \xi' $ converges to $\xi$ uniformly for all $\xi'\in U$.

\end{lemma}
\begin{proof} 
Let $\xi$ and $\eta$ be as in the statement.
Fix any $\E_0\in \yg$.
By the definition of $\La_{\psi}^{\spadesuit}$, there exist $\check\xi\in\La$ and $m\in M$ such that $\Ga(\xi,\check\xi,0,m)\bb R^+$ is dense in $Z_{\cal E_0}$.
Note that $(\xi_0,\eta,0,m)\in Z_{\cal E_0}$ by Lemma \ref{lem.flat}.  Therefore
there exist sequences $\ga_i\in\Ga$ and $t_i\to +\infty$ such that
\begin{align*}
&
\lim_{i\to \infty} \ga_i^{-1}.(\xi,\check\xi,0+t_i,m)
\\
&=\lim_{i\to \infty} (\ga_i^{-1}\xi,\ga_i^{-1}\check\xi,\psi(\log \beta_\xi^A(\ga_i,e))+t_i,\beta_\xi^M(\ga_i,e)m)=(\xi_0,\eta,0,m).
\end{align*}
The last two conditions in \eqref{pp}  immediately follow from this and the first condition follows from \cite[Lem. 8.9]{LO}.

By passing to a subsequence, we may write $\ga_i=k_ia_i\ell_i^{-1} $ where $k_i\to k_0,\ell_i\to\ell_0$ in $K$ and $a_i\in A^+$.
As $\Ga$ is Anosov, $a_i\to\infty$ regularly in $A^+$.
We then have $\ell_0^-=\eta$. Note that $\ga_i\xi'\to k_0^+$ for all $\xi' \in \F $ with $(\xi', \eta) \in \F^{(2)}$ and this convergence is uniform on a compact subset
of $\{\xi': (\xi', \eta)\in \F^{(2)}\}$.
Since $(\xi_0, \eta)\in \F^{(2)}$, there exists a neighborhood $U$ of $\xi_0$ such that $\ga_i\xi' \to k_0^+$ uniformly for all $\xi'\in U$.
Since $\ga_i^{-1}\xi\to \xi_0$ and hence $\ga_i^{-1}\xi\in U$ for all large $i$, we have $\ga_i(\ga_i^{-1}\xi)\to k_0^+$. Hence $\xi=k_0^+$. The claim follows.
\end{proof}

\section{Equi-continuous family of Busemann functions}
We fix a left $G$-invariant and right $K$-invariant Riemannian metric $d$ on $G$. 
 For a subgroup $H<G$ and $\e>0$, we set $H_\e=\{h\in H: d(e,h)<\e\}$. We will use the notation $H_{O(\e)}$ to mean $H_{c\e}$ for some absolute constant $c>0$.
  Recall the notation $o=[K]\in G/K$.

In this section, we prove the following proposition.
\begin{prop}[Equi-continuity] \label{lem.diam} Let $\Gamma <G$ be an Anosov subgroup.
Fix $g\in N^+P$ be such that $g^\pm\in\La$. 
Let $\ga_n\in \Ga$ be a sequence such that
for some $\xi\in\La-\{g^-\}$, $\ga_n^{-1}\xi\to g^+$ and $\ga_n^{-1}g(o)\to g^-$ as $n\to\infty$.
Then, up to
passing to a subsequence of $\ga_n$,
the sequence of maps $\eta\mapsto \beta_\eta^{AM}(\ga_n^{-1} g, g)$
is equi-continuous at $g^+$, i.e.,
 for any $\e>0$, there exists a neighborhood $U_\e$ of $g^+$ in $ \F$ such that for all $n\ge 1$ and for all  $\eta \in U_\e$, 
 $$ \beta_{\eta}^{AM}(\ga_n^{-1}g,g)\subset \beta_{g^+}^{AM}(\ga_n^{-1}g,g)(AM)_{\e}   .$$ 
\end{prop}

We first prove the following two lemmas using the structure theory of semisimple Lie groups.
\begin{lemma}\label{lem.WF}
There exists $c>0$ such that for all sufficiently small $\e>0$, 
 $$aG_\e \subset K_{c\e}aA_{c\e}N\quad\text{ for all $a\in A^+$.}$$
\end{lemma}
\begin{proof}
For all sufficiently small $\e>0$, we have
$$
G_\e\subset M_{O( \e)}N_{O(\e)}^+A_{O(\e)}N_{O(\e)} \text{ and }N_\e^+\subset K_{O(\e)}A_{O(\e)}N_{O(\e)}.
$$
Since $aN_\e^+a^{-1}\subset N_{\e}^+$ for any $a\in A^+$, it follows that
\begin{align*}
aG_\e&\subset  aM_{O(\e)}N_{O(\e)}^+A_{O(\e)}N_{O(\e)}=  M_{O(\e)}(aN_{O(\e)}^+a^{-1})aA_{O(\e)}N_{O(\e)}\\
&\subset M_{O(\e)}(K_{O(\e)}A_{O(\e)}N_{O(\e)})aA_{O(\e)}N_{O(\e)}\subset K_{O(\e)}aA_{O(\e)}N,
\end{align*}
which was to be proved.
\end{proof}

\begin{lemma}\label{lem.M}
Let $g_n=k_na_n\ell_n^{-1}\in KA^+K$ where $a_n\to\infty$ regularly in $A^+$ and $k_n\to k_0$, $\ell_n\to\ell_0$ in $K$ as $n\to\infty$. Assume that both $\xi:=k_0^+$ and $ \zeta:=\ell_0^+$ belong to $N^+e^+$, and set $m_0=m_0[k_0, \ell_0]$ to be
 $$m_0:=k_{\xi}^{-1}k_0 \ell_0^{-1} k_{\zeta}\in M$$
 where $k_\xi, k_\zeta\in K$ are defined as in \eqref{kxi}.
 Then for all small $\e>0$, there exist neighborhoods $V_\e'$ and $U_\e'$ of $\xi$ and $\zeta$, respectively, such that 
$$\{\beta_\eta^{AM}(g_n^{-1},e): \eta \in U_\e'\cap g_n^{-1}V_\e'\}  \subset a_nm_0 (AM)_\e$$
 for all sufficiently large $n>1$.
\end{lemma}
\begin{proof}
By the continuity of the visual maps, there exist neighborhoods $V_\e'$ of $\xi$ and $U_\e'$ of $\zeta$ such that 
$k_\eta\in k_{\zeta} K_{\e}$ for all $\eta\in U_\e'$ and $k_{\eta}\in k_{\xi} K_{\e}$ for all $\eta\in V_\e'$.
We may assume without loss of generality that $k_0^{-1}k_n,\, \ell_n^{-1}\ell_0\in K_{\e}$ for all $n\ge 1$.
Let $\eta\in U_\e'\cap g_n^{-1}V_\e'$ be arbitrary.
By definition,
$$
g_n k_{\eta}\in k_{g_n\eta}
\sigma^{AM}
(g_n,\eta)N, \text{ i.e., } k_0^{-1} g_n k_{\eta}\in k_0^{-1} k_{g_n\eta}
\sigma^{AM}
(g_n,\eta)N.
$$

Observe that
\begin{align*}
k_0^{-1}g_n k_{\eta}&\in k_0^{-1}g_n k_{\zeta}K_{\e}=(k_0^{-1}k_n)a_n(\ell_n^{-1} \ell_0)\ell_0^{-1} k_{\zeta}K_{\e}\\
&\subset K_{\e}a_n K_{\e}\ell_0^{-1} k_{\zeta}K_{\e}\subset K_{\e}a_n K_{O(\e)}\ell_0^{-1} k_{\zeta}.
\end{align*}
On the other hand, since $g_n\eta\in V_\e'$, 
\begin{align*}
&k_0^{-1}g_n k_{\eta}\in
k_0^{-1} k_{g_n\eta}
\sigma^{AM}
(g_n,\eta)N \\& \subset k_0^{-1} k_{\xi}K_{\e}
\sigma^{AM}
(g_n,\eta) N\subset K_{O(\e)}k_0^{-1}k_{\xi}
\sigma^{AM}
(g_n,\eta) N.
\end{align*}
Combining these with the fact that $\ell_0^{-1} k_{\zeta}\in M$, we get
$$
a_n K_{O(\e)}  \cap K_{O(\e)}k_0^{-1} k_{\xi} 
\sigma^{AM}
(g_n,\eta)(\ell_0^{-1} k_{\zeta})^{-1} N\neq \emptyset.
$$
Since $k_0^{-1} k_{\xi}\in M$ as well, it follows from Lemma \ref{lem.WF} that 
\begin{align*}
\sigma^{A}(g_n,\eta)&\in a_n A_{O(\e)},\text{ and }\\
\sigma^{M}(g_n,\eta)&\in (k_0^{-1} k_{\xi})^{-1}M_{O(\e)}\ell_0^{-1} k_{\zeta}\subset (k_0^{-1}k_{\xi})^{-1}\ell_0^{-1}k_{\zeta}M_{O(\e)}.
\end{align*}
Since $\beta_\eta^{AM}(g_n^{-1},e)=\sigma^{AM }(g_n, \eta)$,
and $m_0:=(k_0^{-1}k_{\xi})^{-1}\ell_0^{-1} k_{\zeta}$, this implies the claim.
\end{proof}

\noindent{\bf Proof of Proposition \ref{lem.diam}:}
Set $g_n:=g^{-1}\ga_ng$.
Then $g_n^{-1}(g^{-1}\xi)\to e^+$ and $g_n^{-1}(o)\to e^-$ as $n\to\infty$. By passing to a subsequence, we may write $g_n=k_na_n\ell_n^{-1}\in KA^+K$
where the sequences $k_n$ and $\ell_n$
converge to some $ k_0$ and $\ell_0$ in $K$ respectively. Since $\Ga$ is Anosov, it follows that
$a_n\to\infty$ regularly in $A^+$.
Combined with the hypothesis $g_n^{-1}(o)\to e^-$ as $n\to\infty$, we have $\ell_0^-=e^-$, or equivalently, $\ell_0\in M$. Hence $\ell_0^+=e^+$.

We claim that $k_0^+=g^{-1}\xi$.
Since $a_n\to\infty$ regularly in $A^+$, for any $\eta\in N^+e^+$, $g_n\eta\to k_0^+$ as $n\to\infty$ and the convergence is uniform on a compact subset of $N^+e^+$.
Since $g_n^{-1}(g^{-1}\xi)\to e^+$ as $n\to\infty$, $g_n^{-1}(g^{-1}\xi)$ is contained in a compact subset of $N^+e^+$ for all large $n$, it follows that
$g_n(g_n^{-1}(g^{-1}\xi))\to k_0^+$ as $n\to\infty$, which proves the claim.

Now let $\e>0$ be arbitrary.
Since $g^-\in\La$, by Lemma \ref{ant}, $g^{-1}\La-\{e^-\}\subset N^+e^+.$
Hence both $e^+$ and $ g^{-1}\xi$ belong to $ N^+e^+$.
Applying Lemma \ref{lem.M} to the sequence $g_n$, we obtain $m_0=m_0[k_0, \ell_0]\in M$, and some bounded neighborhoods $U_\e', V_\e' \subset N^+e^+$ of $e^+$ and $g^{-1}\xi$ respectively, such that $$\beta_{\eta'}^{AM} (g_n^{-1},e)\in a_nm_0(AM)_{\e/2} \quad\text{for all $\eta'\in U_\e'\cap g_n^{-1}V_\e'$.}$$
Since $k_0^+=g^{-1}\xi\in V_\e'$ and $U_\e'\subset N^+e^+$, and hence $U_\e'\times \{\ell_0^-\}\subset \F^{(2)}$, we have $g_nU_\e'\subset V_\e'$, and hence
$ U_\e'= U_\e'\cap g_n^{-1}V_\e'$  for all large $n\gg 1$.
Set $U_\e:=gU_\e'\cap N^+e^+$.
Note that $g^+\in U_\e$.

Let $\eta\in U_\e$.
Then $g^{-1}\eta\in U_\e'= U_\e'\cap g_n^{-1}V_\e'$ and hence
\begin{equation}\label{eq.EC}
\beta_{g^{-1}\eta}^{AM}(g_n^{-1},e)\in a_nm_0(AM)_{\e/2}.
\end{equation}

Since $g^{-1}\ga_n\eta=g_n(g^{-1}\eta)\in k_na_n\ell_n^{-1}U_{\e}'$,
we have $g^{-1}\ga_n\eta\to k_0^+\in N^+e^+$, and hence
$g^{-1}\ga_n\eta \in N^+e^+$ for all large $n\gg 1$.
Therefore for all sufficiently large $n>1$,
$\beta_{\eta}^{AM}(\ga_n^{-1}g,g)$ is well-defined and
$$\beta_{\eta}^{AM}(\ga_n^{-1}g,g)= \beta_{g^{-1}\eta}^{AM}(g^{-1} \ga_n^{-1} g,e)=\beta_{g^{-1}\eta}^{AM}(g_n^{-1},e).$$
Hence the lemma follows from the inclusion \eqref{eq.EC}.

\section{Essential values and ergodicity}
As before, we let $\G<G$ be an Anosov subgroup such that $\Ga\cap \inte A^+ M\ne \{e\}$.
Fixing $\psi\in \dg$, let $\nu=\nu_\psi$ be the unique $(\Gamma,\psi)$-Patterson Sullivan measure on $\La$. By Corollary \ref{full1},
\be\label{nf} \nu (N^+e^+\cap \La)=1.\ee

Fix a Borel isomorphism $G/N\to \F\times AM$ given by
\begin{equation}\label{eq.hopf3}
gN\mapsto (g^+,\beta^{AM}_{g^+} (e,g)) \quad
\text{for $g\in N^+AM$}.\end{equation}
This isomorphism is $G$-equivariant for a Borel $G$-action on $\F\times AM$ given by
$$
g(\xi,am)=(g\xi,\beta^{AM}_{\xi}(g^{-1},e)am)
$$
 for $am\in AM$, $g\in G$, and $\xi\in N^+e^+ $ with $ g\xi\in N^+e^+$.

The following then defines a $\Ga$-invariant locally finite measure on $G/N$ by
\begin{equation}\label{eq.nutil2}
d\hat\nu([g])=d\nu(g^+)e^{\psi(\log a)}\,da\,dm
\end{equation}
 where $da$ and $ dm$ are Haar measures on $A$ and $M$ respectively.

Motivated by the work of Schmidt \cite{Sch2} (also \cite{Ro}), we define:
\begin{definition}\label{def.ess}
An element $am\in AM$ is called a $\nu$-\textit{essential value}, if for any Borel set $B\subset \cal F$ with $\nu(B)>0$ and any $\e>0$, there exists $\ga\in\Ga$ such that 
\be\label{ess}
\nu \{\xi\in B\cap\ga^{-1} B : \beta^{AM}_{\xi}(\ga^{-1}, e)\in am(AM)_\e\} >0.
\ee
\end{definition}
In view of $\eqref{nf}$, it suffices to consider Borel subsets $B\subset N^+e^+$ in this definition, and hence
$\beta^{AM}_{\xi}(\ga^{-1}, e)$ is well-defined for all $\xi\in B\cap \ga^{-1}B$.

Let ${\mathsf E}_{\nu}$ denote the set of all $\nu$-essential values in $AM$.
By the following lemma,  $am\in \mathsf E_\nu$ if and only if $(am)^{-1}\in \mathsf E_\nu$;
hence the condition $\beta^{AM}_{\xi}(\ga^{-1}, e)\in am (AM)_\e$ in \eqref{ess} can be replaced by
$\beta^{AM}_{\xi}(e, \ga^{-1})\in am (AM)_\e$ in the above definition.

\begin{lem}
${\mathsf E}_\nu$ is a closed subgroup of $AM$.
\end{lem}
\begin{proof}
Since the metric $d$ restricted to $M$ is bi-$M$-invariant, we have
that for all $\e>0$, $M_\e^{-1}=M_\e$,  $m^{-1} M_{\e} m=  M_\e$ for all $m\in M$ and
$M_{ \e/2}M_{\e/2}\subset M_\e$.
Let $b_1, b_2\in\mathsf E_\nu$.
Let $B\subset \F$ be a Borel subset with $\nu(B)>0$ and let $\e>0$.
Since $b_i\in \mathsf E_\nu$ for $i=1,2$,  there exists $\ga_i\in\Ga$ such that
\begin{align*}
B_1&:=\{\xi \in B\cap\ga_1^{-1} B: \beta_\xi^{AM}(\ga_1^{-1}, e) \in b_1(AM)_{\e/2}\};\\
B_2&:=\{\xi \in B_1\cap\ga_2^{-1} B_1: \beta_\xi^{AM}(\ga_2^{-1}, e) \in b_2(AM)_{\e/2}\}    
\end{align*}
has a positive $\nu$-measure.
Note that 
$B_2\subset B\cap \ga_2^{-1}\ga_1^{-1}B$
and that for all 
$\xi\in B_2$
, we have
\begin{align*}\beta_{\xi}^{AM} (\ga_2^{-1}\ga_1^{-1}, e)&= \beta_{\ga_2\xi}^{AM}(\ga_1^{-1}, \ga_2)
= \beta_{\ga_2\xi}^{AM} (\ga_1^{-1}, e)\beta_{\xi}^{AM} (\ga_2^{-1},e)\\
&\in b_1 (AM)_{\e/2}b_2(AM)_{\e/2} \subset b_1b_2 (AM)_\e . \end{align*} 
Hence
$b_1b_2\in{\mathsf E}_\nu$.
This proves that ${\mathsf E}_\nu$ is a subgroup of $AM$. Now suppose that
a sequence $b_i\in {\mathsf E}_\nu$ converges to some $ b\in AM$. Let $\e>0$ and $B\subset \F$ be a Borel subset with $\nu(B)>0$.
Fix $i$ large enough so that $b_i(AM)_{\e/2}\subset  b (AM)_\e$, and let $\ga_i\in \G$ be such that
$\nu \{\xi \in B\cap\ga_i^{-1} B: \beta_\xi(\ga_i^{-1}, e) \in b_i(AM)_{\e/2}\}>0$.
Then
$\nu \{\xi\in  B\cap\ga_i^{-1}B : \beta_\xi(\ga_i^{-1}, e) \in b(AM)_{\e}\}>0$.
This proves that $b\in \mathsf E_\nu$. Hence ${\mathsf E}_\nu$ is closed.
\end{proof}

\begin{lemma}\label{lem.inv1}  Let $b_0\in \mathsf E_\nu$ be such that
 $\{b b_0b^{-1}: b\in AM\}\subset \mathsf E_\nu$. Then for any $\Ga$-invariant Borel function
 $h : G/N\to[0,1]$, we have
$$h(xb_0)=h(x)\quad\text{ for $\hat\nu$-a.e. $x$.}$$
\end{lemma}
\begin{proof} In view of  the homeomorphsim $N^+AMN/N\to N^+e^+ \times AM$ given by $gN\mapsto (g^+, \beta_{g^+}(e,g))$ and
 \eqref{nf}, it suffices to show that for
any $\G$-invariant Borel function $h:N^+e^+\times AM\to [0,1]$, $h(\xi, b )=h(\xi , bb_0)$ for $\nu$-a.e. $\xi$ and for all $b\in AM$.
Suppose not. Then there exists $b_1\in AM$ such that
$\nu\{\xi\in \F : h(\xi, b_1)<
h(\xi, b_1b_0)\}>0$ or $\nu\{\xi\in \F : h(\xi, b_1)< h(\xi, b_1b_0)\}>0$. We consider the first case; the second case can be treated similarly.
Then there exist $r,\e>0$ such that
$$
Q_{b_0}:=\{\xi\in N^+e^+ : h(\xi, b_1 )<r-\e<r+\e<h(\xi, b_1b_0)\}
$$
has a positive $\nu$-measure.
By considering the convolution of $h$ with the approximation of identity functions on $AM$,
we may assume without loss of generality that the family $h(\xi, \cdot)$, $\xi\in N^+e^+$,
 is uniformly equi-continuous on $AM$. Hence there exists $\e'>0$ such that 
 for all $\xi\in Q_{b_0}$ and $b\in (AM)_{\e'}$,
\be\label{hi} h(\xi, b_1b )<r<h(\xi, b_1b_0b).\ee

Since $b_1b_0b_1^{-1}\in\mathsf E_\nu$ by the hypothesis
and  $\nu(Q_{b_0})>0$, there exists $\ga\in\Ga$ such that
$$
\cal Q:=\{\xi\in Q_{b_0}\cap \ga^{-1} Q_{b_0}: \beta_{\xi}(\ga^{-1},e)\in b_1b_0b_1^{-1} (AM)_{\e'/2}\}
$$
has a positive $\nu$-measure.
We now claim that  
$$h(\xi, b_1 b)<r< h(\gamma (\xi, b_1b))$$ for all $\xi\in \cal Q$ and for
all $b\in (AM)_{\e'/2}$. This yields a contradiction to the $\Gamma$-invariance of $h$.
Since $\cal Q\subset Q_{b_0}$, we have $h(\xi, b_1 b)<r$ for all $b\in (AM)_{\e'}$ by \eqref{hi}.
On the other hand,
for all $b\in (AM)_{\e'/2}$ and $\xi\in \cal Q$, we have
$$\beta_\xi(\ga^{-1}, e) b_1b\in b_1 b_0 b_1^{-1} (AM)_{\e'/2} b_1 b \subset
b_1b_0 (AM)_{\e'} ,$$
since $m^{-1} M_{\e'/2} m M_{\e'/2} \subset M_{\e'}$ for all $m\in M$.
Since $\ga\xi\in Q_{b_0}$ and  $\gamma (\xi, b_1b)=(\gamma\xi, \beta_\xi(\ga^{-1}, e) b_1b)$,
it follows from \eqref{hi} that $h(\gamma (\xi, b_1b))>r$. This proves the claim.
\end{proof}

\section{$N$-ergodic decompositions of BR-measures}
Let $\Ga<G$ be an Anosov subgroup. We prove Theorem \ref{main}(2) in this section.

\subsection{Ergodic decomposition of an infinite measure}
The following version of ergodic decomposition of any Radon measure can be deduced from \cite[Thm.~5.2]{GS}. 
\begin{prop}[Ergodic decomposition]\label{prop.ED}
Let $G$ be a 
locally compact second countable group.
Let $N<G$ be a closed subgroup and $M<G$ be a compact subgroup
normalizing $N$.
Suppose that $NM$
acts
continuously
on a locally compact, $\sigma$-compact,
standard Borel space $(X,\cal B)$, preserving a Radon measure $\mu$ on $X$.
\begin{enumerate}
\item
There exists a Borel map $x\mapsto \mu_x$  
from $X$ to the space of $N$-invariant ergodic Radon measures on $X$ and an $M$-invariant probability measure $\mu^*$ on $X$
 equivalent to $\mu$ with the following properties:
\subitem
(a) $\mu_x=\mu_{xn}$ for every $x\in X$ and $n\in N$.
\subitem
(b) For all nonnegative Borel function $f : X\to\bb R$, we have 
$$
\int f\,d\mu_x=\bb E_{\mu^*}\left( f\frac{d\mu}{d\mu^*}|\cal S_N \right)(x)\qquad \text{for }\mu\text{-a.e. }x\in X,
$$
where $\cal S_N:=\{B\in\cal B :  B.n=B\text{ for all }n\in N\}$.
In particular, we have $$\mu=\int_{x\in X} \mu_x\,d\mu^*(x).$$
If $\mu$ is finite, we can take $\mu^*=\mu$.
\item
Let $\cal T\subset\cal S_N$ be the smallest $\sigma$-algebra such that the map $x\mapsto\mu_x$ is $\cal T$-measurable.
Then $\cal T$ is countably generated, $\cal T=\cal S_N$ mod $\mu$, $\mu_x([y]_{\cal T})=0$ for all $y\not\in[x]_{\cal T}$, and  $\mu_x([x]_{\cal T}^c)=0$ for all $x,y\in X$.
Here $[y]_{\cal T}=\cap_{y\in C\in \cal T} C$ denotes the atom of $y$ in $\cal T$.
\item For each $m\in M$, we have $\mu_{xm}=\mu_x.m$ for $\mu$-a.e. $x\in X$.

\end{enumerate}
\end{prop}
\begin{proof}
Fix an $M$-invariant positive function $\varphi\in L^1(\mu)$ with $\int\varphi\,d\mu=1$.
Then  $d\mu^*:=\varphi\,d\mu$ defines an $N$-quasi-invariant and $M$-invariant
 probability measure on $X$.
By applying \cite[Thm.~5.2]{GS} to $\mu^*$ with the cocycle $\rho : N\times X\to\bb R$ given by $\rho(n,y)=\log \frac{\varphi(yn^{-1})}{\varphi(y)}$, we get a Borel map $x\mapsto \mu^*_x$ from $X$ to the space of $N$-ergodic probability measures such that for all nonnegative Borel function $f : X\to\bb R$, we have
$$
\int f\,d\mu^*_x=\bb E_{\mu^*}(f|\cal S_N)(x) \quad\text{for $\mu^*$-a.e. $x\in X$},
$$
 and 
 $\frac{d(n. \mu^*_x)}{d\mu^*_x}(y)=\frac{\varphi(yn^{-1})}{\varphi(y)}$.
In particular, we have $\mu^*=\int\mu^*_x\, d\mu^*(x)$.
Now define a Radon measure $\mu_x$ on $X$ by $d\mu_x:=\frac{1}{\varphi}\,d\mu^*_x$.
A direct computation shows that $\mu_x$ is $N$-invariant, ergodic for all $x\in X$ and $(1)$ holds.  $(2)$ follows from
the corresponding statement on $\mu^*_x$ from \cite[Thm.~5.2]{GS}.

In order to prove $(3)$,  we compute that
for a non-negative Borel function $f : X\to\bb R$, 
\begin{align*}
\mu^*_{xm}(f)=\bb E_{\mu^*}(f|\cal S_N)(xm)=\bb E_{\mu^*}(m.f|\cal S_N)(x)=\mu^*_x(m.f);
\end{align*}
 the second equality follows since $ \cal S_N.m=\cal S_N$  and $\mu^*$ is $M$-invariant.
It follows that $\mu^*_{xm}=\mu^*_x.m$ for $\mu$-a.e. $x\in X$; this implies (3).
\end{proof}

\subsection{$\pc$-semi-invariant measures}
In terms of the coordinates $G=G/\pc \times A\mc N$,
we have
\be\label{brbr} d\tilde m_\psi^{\BR}  =d\tilde \nu_\psi  e^{\psi(\log a)} dadm dn.\ee

Recall that a measure $\mu$ on $\G\ba G$ is $\pc$-semi-invariant if
there exists a character $\chi:P\to \br_+$ such that for all $p\in \pc$,
$p_*\mu=\chi(p)\mu.$ Since $\chi$ must be trivial on $N\mc$,
 $\mu$ is necessarily $N\mc$-invariant and if we set $\chi_\mu \in \fa^*$
 to be $-\log (\chi|_A)$, we get  that for all $a\in A$,
$$a_*\mu=e^{-\chi_\mu (\log a) }\mu.$$ We set $\psi_\mu:=\chi_\mu+2\rho\in\fa^*$.

\begin{proposition}\label{prop.NMA} Let $\mu$ be
a $\pc$-semi invariant and $N$-ergodic Radon measure  supported on $\cal E$.  
Let $\tilde \mu$ denote its $\Ga$-invariant
lift to $G\simeq G/P^\circ\times AM^\circ N$.
Then $\psi_\mu\in \dg$ and $d\tilde \mu$ is proportional to $d\tilde\nu_{\psi_\mu}|_{\La_0}  e^{\psi_\mu(\log a)} da\,dm\, dn$
for some $\Ga$-minimal subset $\La_0\in \cal Y_\Ga$, or equivalently,
$\mu$ is proportional to $m_{\psi_\mu}^{\BR}|_{\E_0}$ for some $\E_0\in \yg$.  \end{proposition}

\begin{proof}
Since $\tilde\mu$ is a right $P^\circ$-semi-invariant measure on $G\simeq G/P^\circ\times AM^\circ N$, up to a positive constant multiple,
we have
$$
d\tilde\mu=e^{\tilde\chi(\log a)}d\tilde\nu\,da\,dm\,dn
$$
for some Radon measure $\tilde\nu$ on $G/P^\circ$ and $\tilde\chi\in\frak a^*$ \cite[Proposition 10.25]{LO}.
Since $a_*{\tilde \mu}=e^{-\chi_\mu (\log a) }{\tilde \mu}$, it follows $\tilde\chi=\psi_\mu$.
Denote by $\pi : G/P^\circ\to G/P$ the projection map.
Since $\tilde\mu$ is right $N$-ergodic,  $\tilde\nu$ is a $\Ga$-ergodic measure on $G/P^\circ$.
And since $\tilde\mu$ is $\Ga$-invariant, $\pi_*\tilde\nu$ is a $(\Ga,\psi_\mu)$-conformal measure on $G/P$ (cf. \cite[Prop. 10.25]{LO}).
In particular, $\psi_{\mu}\in D_{\Ga}^\star$ by \cite[Thm.~7.7]{LO}.
Let $\tilde\nu_{\psi_\mu}$ be the $M$-invariant lift of $\nu_{\psi_\mu}:=\pi_*\tilde\nu$ to $G/P^\circ$.
Since $\tilde\nu\ll\tilde\nu_{\psi_\mu}$ and $\tilde\nu$ is $\Ga$-ergodic, $\tilde\nu$ is proportional to
$\tilde\nu_{\psi_\mu}|_{\La_0}$ for some $\Ga$-minimal subset $\La_0\in \cal Y_\Ga$ by Proposition \ref{lem.AI}. This completes the proof.
\end{proof}

\subsection{Essential values and Ergodicity}

We fix $\psi\in \dg$ for the rest of the section.
Let $\nu_\psi$ be the unique $(\Ga, \psi)$-Patterson Sullivan measure on $\La$.
Let $\mathsf E_{\nu_\psi}$ be the set of essential values as defined in Definition \ref{def.ess}.
\begin{prop}\label{lem.in}
If $M^\circ\subset \mathsf E_{\nu_\psi}$, then for any $\E_0\in \yg$,
 $m_\psi^{\BR}|_{\E_0}$ is $N$-ergodic.
\end{prop}
\begin{proof} 
Let $m_\psi^{\BR}=\int_{X} \mathsf m_x\,\, d\mathsf m^*(x)$ be an $N$-ergodic decomposition as given by Proposition \ref{prop.ED} with $X=\G\ba G$.
Let $f\in C_c(\Gamma\ba G)$ and consider the map $ h(g):=\mm_{[g]}(f)$ for all $[g]\in X$. Note that $h$ defines a $\G$-invariant Borel function on $G/N$. 
Since $\mc$ is a normal subgroup of $AM$,  Lemma \ref{lem.inv1} implies that $h$ is $\mc$-invariant for $\hat \nu_\psi$-almost all. By Proposition \ref{prop.ED}(3),  it follows that $\mc< \op{Stab}_M (\mm_x)$ for almost all $x$; without loss of generality, we may assume
that $\mc<\op{Stab}_M (\mm_x)$ for all $x\in X$. Hence the finite group $S:=\mc\ba M$ acts on
$\{ \mm_x:x\in  X\}$.
Set $$\tilde\mm_x:=\frac{1}{[M:\mc]}\sum_{s\in \mc\ba M}\mm_{x}.s.$$
Since $m_\psi^{\BR}$ is $M$-invariant, we have $m_\psi^{\BR}=\int_X \tilde \mm_x d\mm^*(x)$.
As $\mm_{xm}=\mm_x.m$ for all $m\in M$,
the map $x\mapsto \tilde\mm_x$ is $NM$-invariant.
Since $m_\psi^{\BR}$ is $NM$-ergodic, $\tilde\mm_x$ is constant $\mm$-a.e. $x\in X$.
Therefore we may fix $x_0\in X$ so that $m_\psi^{\BR}=\tilde \mm_{x_0}$.
Set $M_*:=\op{Stab}_M(\mm_{x_0})$. Then 
$$m_\psi^{\BR}= \frac{1}{[M:M_*]} \sum_{s\in M_*\ba M} \mm_{x_0}.s
$$
where $\mm_{x_0}.s$ are mutually singular to each other. 
We claim that each $\mm_{x_0}.s$ is $A$-semi-invariant with $\psi_{\mm_{x_0}.s}=\psi$ for each $s\in M_*\ba M$.
It suffices to consider the case when $s=[M^*]$.
Let $$A':=\{a\in A: a\text{ preserves the measure class of } \mm_{x_0}\}.$$ As $A'$ is a closed subgroup of $A$,
it suffices to show that for any unit vector $u\in \fa$ and any $\e>0$, $\exp tu\in A'$ for some  $0<t<\e$.
Let $a= \exp \frac{\e u}{n+2} $ for $n=\# M/M^*$. Since $m^{\BR}_\psi$ is quasi-invariant under $a$ and has $n$ number of
ergodic components, it follows that for some $1\le k\le n+1$, $a^k.\mm_{x_0}$ is in the same measure class as $\mm_{x_0}$,
implying that $a^k\in A'$. Hence $A=A'$. As  $m^{\BR}_\psi$ is semi-invariant under $A$, the claim follows.
Therefore, by Proposition \ref{prop.NMA}, $\mm_{x_0}$ is proportional to $m^{\BR}_\psi|_{\E_0}$ for some $\E_0\in \yg$.
 Hence $M_*= \op{Stab}_M m^{\BR}_\psi|_{\E_0}=M_\Ga$.
Since the measures $\mm_{x_0}.s$ are mutually singular to each other, all $\E_0$'s are distinct.
Therefore $m_\psi^{\BR}=\sum_{\E_0\in \yg} c(\E_0)  \cdot m_\psi^{\BR}|_{\E_0}$ for some constant $c(\E_0)>0$.
It remains to observe $c(\E_0)=1$ as the supports of $m_\psi^{\BR}|_{\cal E_0}$ are mutually disjoint from each other.
\end{proof}

\noindent{\bf Proof of Theorem \ref{mmm}.}
Let $\cal O_\Ga$ denote the space  of all $N$-invariant ergodic and $P^\circ$-quasi-invariant
 Radon measures supported on $\cal E$, up to constant multiples.
 We write $\yg=\{\E_i: 1\le i\le k\}$ with $k=\#\yg=\#M/M_\Ga$.
 Consider the map $\iota: \dg\times \{1, \cdots, k\}\to \cal O_\Ga$ defined
 by $\iota (\psi, i)= m_\psi^{\BR}|_{\E_i}$.
 By Proposition \ref{lem.in}, $\iota$ is well-defined. Since any measure contained in $\cal O_\Ga$ must be 
 $\pc$-semi-invariant, being $N$-ergodic,
 Proposition \ref{prop.NMA} implies that $\iota$ is surjective.
That $\iota$ is indeed a homeomorphism now follows because the map $\psi\mapsto m_\psi^{\BR}$ is a homeomorphism
between $\dg$ and  the space  of all $NM$-invariant ergodic and $A$-quasi-invariant
 Radon measures supported on $\cal E$, up to constant multiples, as shown in \cite{LO}.
  This implies Theorem \ref{mmm}, as $\dg$ is homeomorphic to  $ \br^{\text{rank}\,G-1}$ \cite{LO}.
 
 \subsection{The largeness of the length spectrum}
 Without loss of generality, we may assume that $\Ga\cap \op{int} A^+M\ne 
 \emptyset
 $ for the rest of section. 
Recall the notation $\Ga^\star$ from \eqref{eq.GS} and 
$\hat \la (g)$ from Definition \ref{genJo}. We will need the following: 
\begin{prop}\label{cor.A}
For any $C>1$,
the closed subgroup of $AM$  generated by $\{\hat\la(\ga_0)\in AM: \ga_0\in\Ga^\star, \psi(\la(\ga_0))>C\}$ contains $AM^\circ$.
\end{prop}

 By Corollary \ref{jordan} applied to $\Gamma_\psi$, this proposition follows from the following lemma.
\begin{lemma}\label{lem.Sch} For any $C>1$,
there exists a Zariski dense subgroup $\Ga_\psi<\Ga$, depending on $C$, such that
 $\Ga_\psi\cap \inte A^+M\ne 
 \emptyset
 $ and 
 $$\psi(\la(\ga))>C\quad\text{ for all $\ga\in\Ga_\psi-\{e\}$.}$$
 In particular, $\hat \la (\Ga_\psi^\star) \subset \{\hat\la(\ga_0)\in AM: \ga_0\in\Ga^\star, \psi(\la(\ga_0))>C\}$.
\end{lemma}
\begin{proof}
Recall that $\Pi$ is the set of all simple roots of $\frak g$ with respect to $\fa^+$.
By \cite[Lem.~4.3(b)]{Ben}, there exist $\e>0$ and 
$\{s_1,s_2\}\subset\Ga$
such that $s_1\in\op{int}A^+M
$, and for each $m\ge1$, 
$s_1^m, s_2^m$
are $(\Pi, \e)$-Schottky generators and
the subgroup 
$\Ga_m=\langle s_1^m,s_2^m \rangle$
is a Zariski-dense $(\Pi,\e)$-Schottky subgroup of $\Ga$  (see \cite[Def. 4.1]{Ben} for terminologies). 

Fix $m> 1$ and let $z\in \lambda(\Gamma_m)-\{0\}$. Then $z=\la (w)$ for some $w=g_1^{n_1}\cdots g_\ell^{n_\ell}$ with
$g_i\in \{s_1^{\pm m}, s_2^{\pm m}\}$, $n_i\in\bb N$, $g_i\neq g_{i+1}^{-1} (i=1,\cdots,\ell)$ where we interpret $g_{\ell+1}:=g_1$; this is because 
 every element of a $(\Pi,\e)$-Schottky group  is conjugate to a word of such form.
By \cite[Lem. 4.1]{Ben}, there exists $R=R(\e)>0$ (independent of $w\in \Ga_1)$ such that
$$\norm{\la(w)-\sum_{i=1}^\ell n_i\la(g_i)}\leq \ell R.$$

Since $\psi(\la(s_j^{\pm1}))>0$ and $\la(s_j^{\pm m})=m\la(s_j^{\pm1})$,
we can choose $m_0\in \mathbb N$ such that 
$$\psi(\la( s_j^{\pm m_0})) > \norm{\psi}R+ C\quad\text{for each $j=1,2$}.$$ 
Set $$\Ga_\psi:=\Ga_{m_0}.$$ 
Then for any $z=\la(w)\in \la(\Gamma_\psi)-\{0\}$ as above, 
\begin{align*}
\psi(z)&\geq \sum_{i=1}^\ell n_i \psi(\la(g_i)) -\norm{\psi} \ell R  \ge  \sum_{i=1}^\ell n_i\bigg(\psi(\la(g_i)) -\norm{\psi}R\bigg)> C.
\end{align*}

The lemma follows.
\end{proof}

\subsection{Proof of Main proposition}
Recall the $\fa$-valued Gromov product on $\La^{(2)}$: for any $\xi\ne\eta$ in $\La$,
$$\cal G(\xi,\eta): = \log \beta^A_{h^+}(e,h)+\op i \log \beta^A_{h^-}(e,h)$$
 for $h\in G$ satisfying that $h^+=\xi$ and $h^-=\eta$.
 For any fixed $p=g(o)\in G/K$, the following
$$d_{\psi, p}(\xi, \eta):=e^{-\psi(\cal G(g^{-1}\xi, g^{-1}\eta))}  \quad\text{for any $\xi\ne\eta$ in $\La$}$$  defines a virtual visual metric on $\La$, satisfying a weak version of triangle inequality \cite[Lem. 6.11] {LO}. For $\xi\in \La$ and $r>0$, set
$$
\bb B_p(\xi,r):=\{\eta \in \La : d_{\psi,p}(\xi,\eta)<r\}.
$$


We recall the following two lemmas:
\begin{lemma}\cite[Lem. 6.12]{LO} \label{inc} 
There exists $N_0(\psi, p)\ge 1$ satisfying the following:
for any finite collection  $\bb B_p (\xi_1, r_1), \cdots, \bb B_p (\xi_n, r_n)$ with $\xi_i\in \La$ and $r_i>0$, there exists a disjoint subcollection $\bb B_p (\xi_{i_1}, r_{i_1}),
\cdots , \bb B_p (\xi_{i_\ell}, r_{i_\ell})$ such that
$$\bb B_p (\xi_1, r_1)\cup \cdots \cup \bb B_p (\xi_n, r_n)\subset
\bb B_p (\xi_{i_1}, 3N_0(\psi,p) r_{i_1})\cup
\cdots \cup \bb B_p (\xi_{i_\ell}, 3N_0(\psi, p) r_{i_\ell}).$$
Moreover, $N_0(\psi, p)$ can be taken uniformly for all $p$ in a fixed compact subset of $G/K$.
\end{lemma}

\begin{lem}\label{comp3}\cite[Lem. 10.6]{LO}. There exists a compact subset $\cal C\subset G$ such that
for any $\xi\in \La$, there exists $g\in \cal C$ such that $g^+=\xi$ and $g^-\in \La$.
\end{lem}

We set $$N_0:=\max_{p\in \cal C (o)} N_0(\psi, p)<\infty$$ with $N_0(\psi, p)$ and $\cal C$
given by Lemmas \ref{inc} and \ref{comp3} respectively.

\medskip

\begin{proposition}[Main Proposition] \label{prop.E} For all $\ga_0\in\Ga^\star$ satisfying $\psi(\la(\ga_0))>\log 3N_0+1$,
we have $\hat\la(\ga_0)\in\mathsf E_{\nu_\psi}$. 
 \end{proposition}
 
 \subsection{ Proof of Theorem \ref{main}(1)}\label{thm1}
 By Propositions \ref{cor.A} and \ref{prop.E},
 $\mathsf E_{\nu_\psi}$ contains $AM^\circ$. Therefore
Theorem \ref{main}(1) follows from
 Proposition \ref{lem.in}.
 
 
 \medskip
The rest of the section is devoted to the proof of Proposition \ref{prop.E}.

\subsection*{Definition of $\cal B_R(\ga_0,\e)$}

We now fix $\e>0$ as well as an element $\ga_0\in\Ga^\star$ such that
$$
\psi(\la(\ga_0))>\log 3N_0+1.
$$
Note that $y_{\ga\ga_0^{\pm 1}\ga^{-1}}=\ga y_{\ga_0^{\pm 1} }$
for all $\ga\in\Ga$
. We can choose $g\in \cal C$ such that $g^+=y_{\ga_0}$ and $g^-\in\La$.
Note that $g^+\in N^+e^+$, as $\ga_0\in\Ga^\star$.
Set 
$$p:=g(o),\;\; \eta:=g^-, \text{ and } \xi_0:=g^+.$$

For any $\xi\in\La-\{\eta,e^-\}$, we claim that there is $R_\e=R_\e(\xi)>0$ such that
$$
\beta_{\xi'}^{AM}(g,e)\in \beta_{\xi}^{AM}(g,e)(AM)_\e $$
for all $\xi'\in\bb B_p(\xi,e^{\psi(\la(\ga_0)+\la(\ga_0^{-1}))+2\norm{\psi}\e} R_\e).$
Indeed, since $e^-\notin \{\xi, g^{-1}\xi\}$,
we have $\xi, g^{-1}\xi\in N^+e^+$ by Lemma \ref{ant}.
The claim follows as the map $\xi'\mapsto\beta_{\xi'}^{AM}(g,e)$ is continuous at $\xi$.

 By \cite[Lem. 6.11]{LO}, 
the family $\{\bb B_p(\xi,r):\xi\in \La, r>0\}$ forms a basis of topology 
in $\La$.
For $\ga\in \Ga$, let $r_g(\ga)$ be the supremum of $r\geq 0$ such that for all $\xi\in\bb B_p(\ga \xi_0,3N_0r)$,
$\beta_\xi^{AM}(g,\ga\ga_0\ga^{-1}g)$ is well-defined and
\be\label{diff}
\beta_{\xi}^{AM}(g,\ga\ga_0\ga^{-1}g)\in\beta_{\ga \xi_0}^{AM}(g,\ga\ga_0\ga^{-1}g)(AM)_\e.
\ee
If $\ga\xi_0\not\in\{e^-, g^-\}$ and hence $\ga \xi_0, g^{-1}\ga\xi_0\in  N^+e^+$,
then $r_g(\ga)>0$.

For each $R>0$, we define the family of virtual balls as follows:
$$\cal B_R(\ga_0,\e)=\{ \bb B_p(\ga \xi_0,r) : \ga\in\Ga, 0<r<\min (R, r_g(\ga)) \}.
$$

We remark that the difference of the definition of  $\cal B_R$ in this paper and our previous paper \cite{LO} lies
in the definition of $r_g(\gamma)$; in \cite{LO}, we used the $A$-valued Busemann function in \eqref{diff} whereas $r_g(\gamma)$ is defined in terms of the $AM$-valued Busemann function here.

\begin{thm}\cite[Thm.~5.3]{LO} \label{concave}
There exists $C=C(\psi,p)>0$ such that for all $\ga\in\Ga$ and $\xi \in \La$,
$$
-\psi(\underline a (p, \ga p))-C
 \le  \psi(\log \beta^A_\xi(\ga p, p))\leq  \psi(\underline a (\ga p,  p))+C.
$$
where $\underline a(p,q):=\mu(g^{-1}h )$ for $p=g(o)$ and $q=h(o)$.
\end{thm}

For $ q\in G/K$ and $r>0$,
the shadow of the ball $B(q,r)$ viewed from $p
=g(o)
\in G/K$ and $\xi\in \cal F$ are respectively defined as
 $$O_r(p,q):=\{gk^+\in \cal F: k\in K,\, gk\inte A^+o\cap  B(q,r)\ne \emptyset\}$$
 where $g\in G$ satisfies $p=g(o)$, and
$$O_r(\xi,q):=\{h^+\in \cal F: h^-=\xi, ho\in  B(q,r)  \}.$$

\begin{lemma}\cite[Lem. 5.7]{LO}  \label{lem.shadow1}
There exists $\kappa>0$ such that for any $p,q\in G/K$ and $r>0$, we have
$$
\sup_{\xi\in O_r(p,q)}\norm{\log \beta^A_\xi(p,q)-\underline a(p,q)}\leq \kappa r.
$$
\end{lemma}

We let $C=C( \psi, p)>0$  and $\kappa>0$ be the constants  given by Theorem \ref{concave}
and  Lemma \ref{lem.shadow1} respectively. Since $\xi_0$ belongs to the shadow $  O_{\e/(8\kappa)}  (\eta,p)$, we can choose $0<s=s(\gamma_0)< R$
small enough such that 
 \be\label{cho}
\bb B_p( \xi_0,e^{\psi(\la(\ga_0)+\la(\ga_0^{-1}))+\frac{1}{2}\norm{\psi}\e+2C}s)\subset 
 O_{\e/(8\kappa)} (\eta,p).
\ee
Next, observe that the map $\xi'\mapsto \beta_{\xi'}(g,\ga_0g)$ is continuous at $\xi_0$, as $g^{-1}\xi_0=e^+\in N^+e^+$.
Hence we may further assume that $s$ is small enough so that
\be\label{eq.cond2}  
\beta_{\xi'}^{AM}(g,\ga_0g)\in \beta_{\xi_0}^{AM}(g,\ga_0g)(AM)_{\e}\quad \text{ for all }\xi'\in\bb B_p(\xi_0,e^{2C}s).\ee

For each $\ga\in\Ga$, set
\begin{align*}
D(\ga\xi_0,r)&:=\bb B_p(\ga \xi_0,\tfrac{1}{3N_0} e^{-\psi(\mu(g^{-1}\ga g) +\mu(g^{-1}\ga^{-1} g))}r)\text{ and }\\
3N_0D(\ga \xi_0, r) &:=\bb B_p(\ga \xi_0,e^{-\psi(\mu(g^{-1}\ga g) +\mu(g^{-1}\ga^{-1} g))}r ).
\end{align*}
Here note that $\underline{a}(\ga^{-1}p,p)= \mu(g^{-1}\ga g)$ and 
$\i \underline{a}(\ga^{-1}p,p)= \mu(g^{-1}\ga^{-1} g)$.
\begin{lemma}\label{lem.WD2}\label{spn} \label{lem.ge} Let $R>0$ and $\xi\in \La-\{\eta\}$.
Let $\ga_i\in \G$ be a sequence such that  $\gamma_i^{-1} p\to \eta$,
$\gamma_i^{-1}\xi\to \xi_0$,
and $\beta_{\xi}^M(\ga_i,e)\to e$ as $i\to\infty$.
Then, by passing to a subsequence,
 the following holds
for all sufficiently small $r>0$: there exists $i_0=i_0(r)>0$ such  that for all $i\ge i_0$, we have
\begin{enumerate}
\item $\xi \in D(\ga_i\xi_0,r)$ 
and
$D(\ga_i\xi_0,r)\in \cal B_R(\ga_0, \e)$; in particular, for any $R>0$,
$$\spa\subset \bigcup_{D\in \cal B_R(\ga_0,\e)} D.$$
\item $\{ \beta_{\xi'}^{AM}(e,\ga_i\ga_0\ga_i^{-1}):  \xi'\in 3N_0D(\ga_i\xi_0,r)\} \subset \hat\la(\ga_0)(AM)_{O(\e)}.$
\end{enumerate}

\end{lemma}
\begin{proof}
Let $g\in G$ be such that $p=g(o)$.
Note that $\ga_i^{-1}go\to\eta=g^-$ and $\ga_i^{-1}\xi\to \xi_0=g^+$. 
By passing to a subsequence, we have
a neighborhood $U_\e\subset \F$ of $\xi_0$ associated to the sequence $\ga_i$ given by Proposition \ref{lem.diam}.
Since $\xi_0\in U_\e$, there exists $R_1>0$ such that
 $$
 \bb B_p(\xi_0,e^{2C}R_1),\ga_0^{-1}\bb B_p(\xi_0,e^{2C}R_1)\subset U_\e.
 $$
Let 
$0<r< \min(s(\ga_0),R_\e/2,R_1,R)$.
In view of \cite[Lem. 10.12]{LO},  we have $3N_0D(\ga_i\xi_0,r) \subset \ga_i\bb B_p(\xi_0,e^{2C}r)$.
In order to show that $D(\ga_i\xi_0,r)\in \cal B_R(\ga_0,\e)$, it suffices to check that for all $\xi'\in \bb B_p(\xi_0,e^{2C}r)$,
$$
\beta_{\xi'}^{M}(\ga_i^{-1}g,\ga_0\ga_i^{-1}g)\in \beta_{\xi_0}^{M}(\ga_i^{-1}g,\ga_0\ga_i^{-1}g)M_{\e};
$$
this implies that $r<r_g(\ga_i)$.

We start by noting that since $r\leq s(\ga_0)$, 
we have $\beta_{\xi'}^{M}(g,\ga_0g)\in \beta_{\xi_0}^{M}(g,\ga_0g)M_{\e}.$ Since $\xi', \ga_0^{-1}\xi'\in U_\e$, by Proposition \ref{lem.diam},
for all sufficiently large $i$,
\begin{align*}
\beta_{\xi'}^{M}(\ga_i^{-1}g,\ga_0\ga_i^{-1}g)&=\beta_{\xi'}^{M}(\ga_i^{-1}g,g)\beta_{\xi'}^{M}(g,\ga_0g)\beta_{\xi'}^{M}(\ga_0g,\ga_0\ga_i^{-1}g)\\
&=\beta_{\xi'}^{M}(\ga_i^{-1}g,g)\beta_{\xi'}^{M}(g,\ga_0g)\beta_{\ga_0^{-1}\xi'}^{M}(\ga_i^{-1}g,g)^{-1}\\
&\in\beta_{\xi_0}^{M}(\ga_i^{-1}g,g)\beta_{\xi_0}^{M}(g,\ga_0g)\beta_{\xi_0}^{M}(\ga_i^{-1}g,g)^{-1}M_{O(\e)}\\
&=\beta_{\xi_0}^{M}(\ga_i^{-1}g,\ga_0\ga_i^{-1}g)M_{O(\e)},
\end{align*}
which verifies that $D(\ga_i\xi_0,r)$ belongs to the family $\cal B_R(\ga_0,\e)$. The claim that
$\xi \in D(\ga_i\xi_0,r)$ can be shown in the same way as in the proof of \cite[Lem. 10.12]{LO}. This proves (1).

(1) implies that for all sufficiently large $i$ 
and $\xi'\in 3N_0D(\ga_i\xi_0,r)$, we have
\begin{equation}\label{eq.g1}
\beta_{\xi'}^{AM}(g,\ga_i\ga_0\ga_i^{-1}g)\in\beta_{\ga_i \xi_0}^{AM}(g,\ga_i\ga_0\ga_i^{-1}g)(AM)_\e.
\end{equation}
Now note that for all $\xi'\in 3N_0D(\ga_i\xi_0,r)$,
\begin{align}\label{eq.g2}
\beta_{\xi'}^{AM}(e,\ga_i\ga_0\ga_i^{-1})&=\beta_{\xi'}^{AM}(e,g) \beta_{\xi'}^{AM}(g,\ga_i\ga_0\ga_i^{-1}g)\beta_{\xi'}^{AM}( \ga_i\ga_0\ga_i^{-1}g,\ga_i\ga_0\ga_i^{-1})\notag\\
&=\beta_{\xi'}^{AM}(e,g) \beta_{\xi'}^{AM}(g,\ga_i\ga_0\ga_i^{-1}g)\beta_{\ga_i\ga_0^{-1}\ga_i^{-1}\xi'}^{AM}( e,g)^{-1}.
\end{align}
On the other hand,
{
\begin{align*}
d_p(\ga_i\ga_0\ga_i^{-1}\xi',\ga_i\xi_0)&=e^{-\psi(\log\beta_{\xi'}^A(\ga_i\ga_0^{-1}\ga_i^{-1}g,g)+\op i\log\beta_{\ga_i\xi_0}^A(\ga_i\ga_0^{-1}\ga_i^{-1}g,g))}d_p(\xi',\ga_i\xi_0)\\
&\leq e^{\psi(\la(\ga_0)+\la(\ga_0^{-1}))+2\norm{\psi}\e}d_p(\xi',\ga_i\xi_0),
\end{align*}
and hence 
$$
\xi',\ga_i\ga_0\ga_i^{-1}\xi'\in\bb B_p(\ga_i\xi_0,e^{\psi(\la(\ga_0)+\la(\ga_0^{-1}))+2\norm{\psi}\e}r ).
$$
Since \be\label{fin5} \ga_i\xi_0\to \xi\quad\text{ as $i\to\infty$}\ee
by Lemma \ref{lem.C2} and $r<R_\e/2$,
}
for all sufficiently large $i$ and all $\xi'\in 3N_0D(\ga_i\xi_0,r)$, the elements 
 $\xi'$, $\ga_i\ga_0\ga_i^{-1}\xi'$, and $\ga_i\xi_0$ all  belong to the subset $\bb B_p(\xi,
{
e^{\psi(\la(\ga_0)+\la(\ga_0^{-1}))+2\norm{\psi}\e}
} R_\e)$.  Hence
\begin{equation}\label{eq.g3}
\beta_{\xi'}^{AM}(e,g), \beta_{\ga_i\ga_0^{-1}\ga_i^{-1}\xi'}^{AM}( e,g),
{
 \beta_{\ga_i\xi_0}^{AM}(e,g)
}
\in \beta_{\xi}^{AM}(e,g)M_\e.
\end{equation}
Combining \eqref{eq.g1}, \eqref{eq.g2} and \eqref{eq.g3}, it follows that for all $\xi'\in 3N_0D(\ga_i\xi_0,r)$,
$$
\beta_{\xi'}^{AM}(e,\ga_i\ga_0\ga_i^{-1})\in\beta_{\ga_i\xi_0}^{AM}(e,\ga_i\ga_0\ga_i^{-1})(AM)_{O(\e)}.
$$

 
Note that by Proposition \ref{lem.diam} and \eqref{fin5}, we get
\begin{align}\label{fin3} 
    \beta_{\xi_0}^{AM}(\ga_i^{-1},e)&=\beta_{\xi_0}^{AM}(\ga_i^{-1},\ga_i^{-1}g)\beta_{\xi_0}^{AM}(\ga_i^{-1}g,g)\beta_{\xi_0}^{AM}(g,e)\notag \\
    &= \beta_{\ga_i\xi_0}^{AM}(e,g)\beta_{\xi_0}^{AM}(\ga_i^{-1}g,g)\beta_{\xi_0}^{AM}(g,e)\notag \\
    &\in \beta_{\xi}^{AM}(e,g)\beta_{\ga_i^{-1}\xi}^{AM}(\ga_i^{-1}g,g)\beta_{\xi_0}^{AM}(g,e)(AM)_{O(\e)}\notag \\
    &=\beta_{\ga_i^{-1}\xi}^{AM}(\ga_i^{-1},\ga_i^{-1}g)\beta_{\ga_i^{-1}\xi}^{AM}(\ga_i^{-1}g,g)\beta_{\ga_i^{-1}\xi}^{AM}(g,e)(AM)_{O(\e)}\notag\\
    &=\beta_{\ga_i^{-1}\xi}^{AM}(\ga_i^{-1},e)(AM)_{O(\e)}
\end{align}

Since
$\beta_{\ga_i^{-1}\xi}^{M}(\ga_i^{-1},e)=\beta_\xi^M(e,\ga_i)\to e$ as $i\to\infty$ by the hypothesis,
\eqref{fin3} implies that
\be\label{fin4}
\beta_{\xi_0}^{M}(\ga_i^{-1},e)\in M_{O(\e)}\text{ for all large enough }i.\ee
Since
\begin{align*}
\beta_{\ga_i\xi_0}^{AM}(e,\ga_i\ga_0\ga_i^{-1})&=\beta_{\ga_i\xi_0}^{AM}(e,\ga_i)\beta_{\ga_i\xi_0}^{AM}(\ga_i,\ga_i\ga_0)\beta_{\ga_i\xi_0}^{AM}(\ga_i\ga_0,\ga_i\ga_0\ga_i^{-1})\\
&=\beta_{\xi_0}^{M}(\ga_i^{-1},e)\hat\la(\ga_0)\beta_{\xi_0}^{M}(\ga_i^{-1},e)^{-1},
\end{align*}
we deduce from \eqref{fin4} that $$\beta_{\xi'}^{AM}(e,\ga_i\ga_0\ga_i^{-1})\in \hat\la(\ga_0) (AM)_{O(\e)}$$ as desired.
\end{proof}

\begin{lemma}\label{lem.sh2}
Let $B\subset \mathcal F$ be a Borel set with $\nu_\psi(B)>0$.
Then for $\nu_\psi$-a.e. $\xi\in B$,
$$
\limsup\limits_{R\to0}\left\{\frac{\nu_\psi (B\cap D(\ga\xi_0,r))}{\nu_\psi (D(\ga\xi_0,r))} : 
\begin{array}{c}
\xi\in D(\ga\xi_0,r), r<R,\text{ and }\\
\beta_{\xi'}^{AM}(e,\ga\ga_0\ga^{-1})\in\hat\la(\ga_0)(AM)_{\e}\\
\text{ for all }\xi'\in 3N_0D(\ga\xi_0,r)
\end{array}
\right\}=1.
$$
\end{lemma}
\begin{proof}
To each Borel function $h: G/P\to\bb R$, we associate a function $h^* : G/P\to\bb R$ defined by
$$
h^*(\xi)=\limsup\limits_{R\to0}\left\{\frac{1}{\nu_\psi (D)}\int_D h\,d\nu_\psi :
\begin{array}{c}
\xi\in D=D(\ga\xi_0,r), r<R,\text{ and }\\
\beta_{\xi'}^{AM}(e,\ga\ga_0\ga^{-1})\in\hat\la(\ga_0)(AM)_{\e}\\
\text{ for all }\xi'\in 3N_0D(\ga\xi_0,r)
\end{array}
\right\}.
$$
By Lemma \ref{lem.C2} and \ref{lem.ge}, $h^*$ is well defined on $\spa-\{\eta\}$ and hence $\nu_\psi $-a.e. on $G/P$ by Corollary \ref{full}.
We may then apply the same argument as in \cite[Proof of Prop. 10.17]{LO} to deduce $h^*=h$  $\nu_\psi $-a.e. Hence the lemma follows by taking $h={\bf 1}_B$. \end{proof}

\subsection*{Proof of Proposition \ref{prop.E}}
Let $B\subset\cal F$ be a Borel set such that $\nu_\psi (B)>0$ and let $\e>0$ be arbitrary.
By Lemma \ref{lem.sh2}, for $\nu_\psi$-a.e. $\xi\in B$, there exist $\ga\in\Ga^\star$ and $D= D(\ga\xi_0,r)\in\cal B_{R}(
\ga_0,\e)$ containing $\xi$ such that
\begin{enumerate}
\item
$\nu_\psi(D\cap B)> (1+e^{-\psi(\la(\ga_0^{-1}))-\norm{\psi}\e})^{-1}\nu_\psi(B)$,\text{ and}
\item
$\beta_{\xi'}^{AM}(e,\ga\ga_0\ga^{-1})\in\hat\la(\ga_0)(AM)_{\e}\text{ for all }\xi'\in 3N_0D(\ga\xi_0,r)$.
\end{enumerate}
We claim that
\begin{equation}\label{eq.sups}
\{\xi\in B\cap\ga\ga_0\ga^{-1} B : \beta_\xi^{AM}(e,\ga \ga_0\ga^{-1})\in \hat\la(\ga_0)(AM)_\e\}
\end{equation}
has a positive $\nu_\psi$-measure, which will finish the proof.

We have $\ga{\ga_0}\ga^{-1} D\subset D$ by \cite[Proof of Prop. 10.7]{LO}.
Together with (2) above, it follows that 
$$
\beta_\xi^{AM}(e,\ga \ga_0\ga^{-1})\in \hat\la(\ga_0)(AM)_\e \quad\text{for all $\xi\in\ga\ga_0\ga^{-1}D$.}
$$

Consequently, \eqref{eq.sups} contains
\begin{equation}\label{eq.subs}
(D\cap B)\cap \ga{\ga_0}\ga^{-1}(D\cap B),
\end{equation}
which has a positive $\nu_\psi$-measure by \cite[Proof of Prop. 10.7]{LO}.
This proves the claim.
$\qed$

\begin{Rmk}\rm
We remark that the approach of this paper shows the following result when $G$ has rank one.
\begin{thm}\label{final}
Let $G$ have rank one, and $\Gamma<G$ be a Zariski dense discrete subgroup. Let $\nu_o$ be an ergodic $\Gamma$-conformal probability measure on the limit set of $\Gamma$.
Let $m^{\BMS}$ and $m^{\BR}$ be respectively the $\BMS$ and $\BR$ measures on $\Gamma\ba G$ associated to $\nu_o$.
Suppose that $m^{\BMS}$ is $AM$-ergodic. Then $m^{\BMS}$ is $A$-ergodic and  $m^{\BR}$ is $N$-ergodic.
\end{thm}
In the rank one case, all the properties that we had to establish for Anosov groups hold automatically from the
negative curvature property of the associated symmetric space. As $\Gamma$ is Zariski dense, Theorem \ref{thm.AE} proves that $m^{\BMS}$ is  the sum of at most $[M:M^\circ]$ number of $A$-ergodic components. Then the Hopf ratio ergodic theorem for the one-parameter subgroup $A$ implies that $\nu_o$ gives full measure on the set of strong Myrberg limit points of $\Gamma$, i.e., Corollary \ref{full3} holds.
Now the arguments in section 7 shows that the set of $\nu_o$-essential values is equal to $AM$, and hence $m^{\BR}$ is
the sum of at most  $[M:M^\circ]$ number of $N$-ergodic components. When  $G\not\simeq \op{SL}_2(\br)$, $M$ is connected
\cite[Lem. 2.4]{Win}
and for $G\simeq \op{SL}_2(\br)$, $M_\Gamma=\{\pm e\}$ by (\cite{CG}, Lem. 2). Hence Theorem \ref{final} follows.
\end{Rmk}

\end{document}